\documentclass[11pt,leqno]{amsart}

\usepackage{amssymb}
\usepackage{amsmath}
\usepackage{amsxtra}
\usepackage{amscd}
\usepackage{mathtools}
\usepackage{float}
\usepackage{graphicx} 
\usepackage{amsfonts}
\usepackage{pb-diagram}
\usepackage[utf8]{inputenc}
\usepackage[T1]{fontenc}
\usepackage{enumerate}
\usepackage{array,multirow}
\usepackage[all]{xy}
\usepackage{cases}
\usepackage{empheq}
\usepackage{amsbsy}
\usepackage{multirow}

\newtheorem{theorem}{Theorem}
\newtheorem{corollary}{Corollary}
\newtheorem{lemma}{Lemma}
\newtheorem{proposition}{Proposition}
\newtheorem{definition}{Definition}
\newtheorem{remark}{Remark}

\numberwithin{equation}{section}

\makeatletter
\let\@wraptoccontribs\wraptoccontribs

\makeatother

\begin{document} 

\title{Framization of the Temperley--Lieb Algebra}

\author{D. Goundaroulis}
\address{Departament of Mathematics,
National Technical University of Athens,
Zografou campus, GR--157 80 Athens, Greece.}
\email{dgound@mail.ntua.gr}
\urladdr{users.ntua.gr/dgound}

\author{J. Juyumaya}
\address{Instituto de Matem\'{a}ticas, Universidad de Valpara\'{i}so, 
Gran Breta\~{n}a 1091, Valpara\'{i}so, Chile}
\email{juyumaya@uvach.cl}

\author{A. Kontogeorgis}
\address{Department of Mathematics, University of Athens\\
Panepistimioupolis, 15784 Athens, Greece
}
\email{kontogar@math.uoa.gr}

\author{S. Lambropoulou}
\address{ Department of Mathematics,
National Technical University of Athens,
Zografou campus, GR--157 80 Athens, Greece.}
\email{sofia@math.ntua.gr}
\urladdr{http://www.math.ntua.gr/~sofia}

\thanks{This research has been co-financed by the European Union (European Social
 Fund -- ESF) and Greek national funds through the Operational Program
 "Education and Lifelong Learning" of the National Strategic Reference
 Framework (NSRF) -- Research Funding Program: THALES: Reinforcement of the
 interdisciplinary and/or inter-institutional research and innovation. Moreover, 
 the second author was partially supported by Fondecyt 1141254 and   DIUV Grant Nº1--2011. 
}



\date{}

\begin{abstract}
We propose a framization  of the Temperley--Lieb algebra. The framization is a procedure that can briefly be described as the adding of framing to a known knot algebra in a way that is both algebraically consistent and topologically meaningful. Our framization of the Temperley--Lieb algebra is defined 
as a quotient of the   Yokonuma--Hecke algebra. The main theorem provides necessary and sufficient conditions for the Markov trace defined on the Yokonuma--Hecke algebra to pass through to the quotient algebra. Using this we construct 1-variable invariants for classical knots and links, which, as we show, are not topologically equivalent to the Jones polynomial.
\end {abstract}
\maketitle

\setcounter{tocdepth}{1}
\tableofcontents
\section{Introduction}
Since the original construction of  the Jones polynomial the Temperley--Lieb algebra has become a  cornerstone of  a fruitful interaction between Knot theory and Representation theory. The  Temperley--Lieb algebra was introduced by Temperley and Lieb \cite{tl71} and was rediscovered by Jones \cite{jo83} as a knot algebra \cite{jo}.

\smallbreak

 A knot algebra is an algebra that is used in the construction of invariants of classical links using Jones' method \cite{jo}. More precisely, a knot algebra ${\rm A}$ is a triplet $({\rm A}, \pi , \tau ) $, where $\pi$ is an appropriate representation of the braid group in ${\rm A}$ and $\tau$ is a Markov trace function defined on ${\rm A}$. The Temperley--Lieb algebra, the Iwahori--Hecke algebra and the BMW algebra are the most known examples of knot algebras.

\smallbreak
 The \lq framization\rq\  of a knot algebra is a mechanism designed by the second and fourth authors, that consists in a generalization of  a knot algebra via the  addition of  framing generators. In this way we obtain a new algebra which is related to framed braids and framed knots. More precisely, the framization procedure can roughly be described as the procedure of adding framing generators to the generating set of  a knot algebra, of defining interacting relations between the framing generators and the original generators of the algebra and of applying framing on the  original  defining  relations of the algebra. The resulting framed relations should be  topologically consistent. The challenge in this procedure is to apply the framization  on the relations  of polynomial type. 

\smallbreak

The basic example of framization is the Yokonuma--Hecke algebra, ${\rm Y}_{d,n}(u)$, which can be regarded as a  framization of the Iwahori--Hecke algebra, ${\rm H}_n(u)$ \cite{jula2, jula}. The quadratic relation of ${\rm Y}_{d,n}(u)$ involves intrinsically the framing generators, while for $d=1$ the algebra ${\rm Y}_{1,n}(u)$ coincides with ${\rm H}_n(u)$. Having in mind this example, the second and fourth authors proposed framizations of several knot algebras \cite{jula5, jula6}. 
\smallbreak
The aim of this paper is to propose a framization of the Temperley--Lieb algebra and to derive from this new algebra knot and link invariants via an appropriate Markov trace. The Temperley--Lieb algebra can be regarded as a quotient of the Iwahori--Hecke algebra. Therefore, it is natural to search for a quotient of the Yokonuma--Hecke algebra over an appropriate two-sided ideal, that can be considered as a framization of the Temperley--Lieb algebra. Although such an ideal is not unique, it will become clear that our choice for the ideal that leads to the framization of the Temperley--Lieb algebra is the most natural one with respect to the construction of related framed and classical link invariants. Indeed, in Section~\ref{motivation} we first discuss two natural  quotients of ${\rm Y}_{d,n}(u)$ that could possibly lead to a framization of the Temperley--Lieb algebra, the Yokonuma--Temperley--Lieb algebra, ${\rm YTL}_{d,n}(u)$ (introduced and studied in \cite{gojukola}) and the Complex Reflection Temperley--Lieb algebra, ${\rm CTL}_{d,n}(u)$. These two quotient algebras, however, are not suitable for our purpose, since: The algebra ${\rm YTL}_{d,n}(u)$ is too restricted and, as a consequence,  the invariants for classical links from the algebra ${\rm YTL}_{d,n}(u)$ just recover the Jones polynomial \cite{gojukola}. On the other hand, as we shall see, the algebra ${\rm CTL}_{d,n}(u)$ is too large for our topological purposes. We proceed with introducing a third quotient of ${\rm Y}_{d,n}(u)$, the Framization of the Temperley--Lieb algebra, ${\rm FTL}_{d,n}(u)$, which lies between ${\rm YTL}_{d,n}(u)$ and ${\rm CTL}_{d,n}(u)$ and which will turn out to be the right one. The connection between all three quotients of ${\rm Y}_{d,n}(u)$ is then analyzed. We note that for $d=1$ all three quotients coincide with the Temperley--Lieb algebra ${\rm TL}_n(u)$. We further provide presentations with non-invertible generators for the quotient algebras ${\rm FTL}_{d,n}(u)$ and ${\rm CTL}_{d,n}(u)$.  Such a presentation for the quotient algebra ${\rm YTL}_{d,n}(u)$ was given in \cite{gojukola}. We conclude this section with a result by Chlouveraki and Pouchin  \cite{ChPou2} regarding the dimensions of the quotient algebras ${\rm FTL}_{d,n}(u)$ and ${\rm CTL}_{d,n}(u)$.
\smallbreak
Returning to our basic example, the Yokonuma--Hecke algebra, the second author has constructed a unique Markov trace function, ${\rm tr}$, on the algebra ${\rm Y}_{d,n}(u)$ with parameters $z, x_1 , \ldots , x_{d-1}$ \cite{ju}. Consequently, invariants for framed, classical and singular oriented links have been obtained \cite{jula,jula4,jula3} by applying the so-called \lq ${\rm E}$--condition\rq\ on the parameters $x_1, \ldots , x_{d-1}$ so that ${\rm tr}$ re-scales according to the negative stabilization move between framed braids \cite{jula}. These invariants, in particular those for classical links, was necessary to be compared with other known invariants, especially with the 2-variable Jones or Homflypt polynomial. In \cite{ChLa} it was proved that these polynomial invariants do not coincide with the Homflypt polynomial, except in trivial cases. Yet they could be topologically equivalent to the Homflypt polynomial, in the sense that they might distinguish the same pairs of non-isotopic links. Eventually, in a recent development \cite{ChJuKaLa}, another presentation for the Yokonuma--Hecke algebra is employed with parameter $q$ in a new quadratic relation, where $q^2=u$ \cite{ChPoulain}. Using this presentation, the authors of \cite{ChJuKaLa} have been able to establish that the classical link invariants, $\Theta_d$, obtained from the isomorphic algebra ${\rm Y}_{d,n}(q)$ coincide with the Homflypt polynomial on {\it knots}, but they are {\it not topologically equivalent to the Homflypt polynomial on links} (as it was conjectured in  \cite{CJJKL}). 
\smallbreak

The next natural question is to examine under what conditions the trace ${\rm tr}$ on the algebra ${\rm Y}_{d,n}(u)$ passes through to the quotient algebras ${\rm FTL}_{d,n}(u)$ and ${\rm CTL}_{d,n}(u)$ respectively. We recall that, in the classical case, as Jones showed, the Ocneanu trace on the Iwahori--Hecke algebra \cite{jo} passes to the quotient ${\rm TL}_n(u)$ if and only if the trace parameter $\zeta$ takes certain specific values.
Accordingly, in Section~\ref{ftlsection} we provide the necessary and sufficient conditions for the Markov trace ${\rm tr}$ \cite{ju} on the  Yokonuma--Hecke algebra to pass through to the quotient algebras  ${\rm FTL}_{d,n}(u)$ and ${\rm CTL}_{d,n}(u)$. The corresponding conditions for the algebra ${\rm YTL}_{d,n}(u)$ are given in \cite{gojukola}. More precisely, we first find the necessary and sufficient conditions on the trace parameters $z, \, x_1 \, \ldots , x_{d-1}$, for the algebra ${\rm FTL}_{d,3}(u)$ using tools from harmonic analysis on finite groups (Lemma~\ref{lemmaN3}) and then we generalize our result using induction on $n$ (Theorem~\ref{akthmgen}). Using the same methods we prove the analogous theorem for ${\rm CTL}_{d,n}(u)$ (Theorem~\ref{ctlthm}). For $d=1$ the specific values we found for $z$ coincide with those found by Jones for ${\rm TL}_n(u)$ \cite{jo}. Finally, we discuss the connections between the necessary and sufficient conditions for ${\rm tr}$ to pass to all three quotient algebras ${\rm CTL}_{d,n}(u)$, ${\rm FTL}_{d,n}(u)$ and ${\rm YTL}_{d,n}(u)$. 
\smallbreak
Using the above conditions on the trace ${\rm tr}$ and subjecting the trace parameters $x_1, \ldots , x_{d-1}$ to the ${\rm E}$--condition, we define in Section~\ref{knotinv} invariants for framed and classical links through the quotient algebras ${\rm FTL}_{d,n}(u)$ and ${\rm CTL}_{d,n}(u)$.  We then show that the invariants from the algebras ${\rm CTL}_{d,n}(u)$ coincide either with some of the invariants from ${\rm Y}_{d,n}(u)$ or with some of the invariants from ${\rm FTL}_{d,n}(u)$. Since ${\rm CTL}_{d,n}(u)$ is larger than ${\rm FTL}_{d,n}(u)$ and since we do not obtain from ${\rm CTL}_{d,n}(u)$ any extra invariants, for these reasons ${\rm FTL}_{d,n}(u)$ is chosen as the framization of the Temperley--Lieb algebra.
\smallbreak
Focusing now on the classical link invariants from the algebra ${\rm FTL}_{d,n}(u)$, these need to be compared to the Jones polynomial. Following \cite{ChJuKaLa}, in Section~\ref{identif} we give a new presentation for the algebra ${\rm FTL}_{d,n}$ with parameter $q$ deriving from the new presentation of the Yokonuma--Hecke algebra ${\rm Y}_{d,n}(q)$. We then adjust our results so far to the isomorphic algebra ${\rm FTL}_{d,n}(q)$ and we apply them to the results of \cite{ChJuKaLa}. Namely, by specializing $\Theta_d(q,z)$  to the our specific value for $z$, we obtain 1-variable invariants for classical knots and links, denoted by $\theta_d(q)$. Finally, adapting the results of \cite{ChJuKaLa} to the invariants $\theta_d(q)$ we show that they coincide with the Jones polynomial on {\it knots} but they are {\it not topologically equivalent to the Jones polynomial on links}.
\smallbreak
The outline of the paper is as follows:
Section~\ref{prelim} is dedicated to providing necessary definitions and results, including: the Iwahori--Hecke algebra, the Ocneanu trace  and the Yokonuma--Hecke algebra. In Section~\ref{sectioninv} we recall some basic tools from harmonic analysis of finite groups, such as the convolution product, the product by coordinates and the Fourier transform, necessary for exploring the \lq ${\rm E}$--system\rq . In Section~\ref{motivation} we discuss three quotients of the Yokonuma--Hecke algebra as possible candidates for the framization of the Temperley--Lieb algebra. In Section~\ref{ftlsection} we provide necessary and sufficient conditions for the ${\rm tr}$ on the Yokonuma--Hecke algebra to pass through to the quotient algebras ${\rm FTL}_{d,n}(u)$ and ${\rm CTL}_{d,n}(u)$. In Section~\ref{knotinv} we define 1-variable framed and classical link invariants related to the algebras ${\rm FTL}_{d,n}(u)$ and ${\rm CTL}_{d,n}(u)$. Finally, in Section~\ref{identif} we prove that 1-variable classical link invariants derived from the isomorphic algebra ${\rm FTL}_{d,n}(q)$ are not topologically equivalent to the Jones polynomial.
\smallbreak
The results of this paper lead to further questions worth investigating, as for example, the possibility of obtaining new 3-manifold invariants related to the invariants $\theta_d$, in analogy to the Witten invariants \cite{wi}.


\section{Preliminaries}\label{prelim}

\subsection{{\it Notation}}
 Throughout the paper by the term algebra we mean an associative  unital (with unity $1$)  algebra over $ \mathbb{C}(u)$, where $u$ is an indeterminate. Thus we can regard   $ \mathbb{C}(u) $ as a subalgebra  
 of the center of the algebra. We will also fix two positive integers, $d$ and $n$. 
\smallbreak 
As usual we denote by ${\mathbb Z}/d{\mathbb Z}$  the group of integers modulo $d$. We will also denote the underlying set of the group $\mathbb{Z}/d\mathbb{Z}$ by $\{0, 1,\ldots, d-1\}$.
\smallbreak
We denote  $S_n$ the symmetric group on the set $\{1, 2, \ldots , n \}$. Let $s_i$ be the elementary transposition $(i, i+1)$ and let
$\langle s_i, s_j\rangle$ denote the subgroup generated by $s_i$ and $s_j$.  We also denote by $l^\prime$ the  length function on $S_n$ with respect to the $s_i$'s.
\smallbreak
Denote by $C$ the infinite cyclic group and by $C_d= \langle t \, | \, t^d=1 \rangle$  the cyclic group of order $d$. Let $t_i := (1,\ldots, 1,t ,1, \ldots , 1 )\in C_d^n$,
where $t$ is in the  $i$-th position. We then have:
\[
C_d^n = \langle t_1, \ldots  ,t_n \, | \, t_i t_j = t_j t_i, \, t_i^d=1 \rangle .
\]

Define $C_{d,n}: = C_d^n \rtimes S_n $, where the action is defined by permutation on the indices of the $t_i$'s, namely: $s_it_j = t_{s_i(j)} s_i$.  Notice that $C_{d,n}$ is isomorphic to the {\it complex reflection group} $G(d,1,n)$. We also introduce the following notation $C_{\infty,n}:= C^n \rtimes S_n$.
\smallbreak
Denote by $B_n$ the braid group of type $A$, that is, the group generated by the elementary braidings $\sigma_1, \ldots , \sigma_{n-1}$, subject to the following relations: $\sigma_i \sigma_j \sigma_i = \sigma_j \sigma_i \sigma_j$, for $|i-j|=1$ and $\sigma_i \sigma_j = \sigma_j \sigma_i$, for $|i-j|>1$. We will also use the {\it $d$-modular framed braid group} ${\mathcal F}_{d,n}:= C_d^n \rtimes B_n$, where the action of $B_n$ on $C_d^n$ is defined by the induced permutation on the indices of the $t_i$'s. We will also refer to the {\it framed braid group} $\mathcal{F}_{n}:= C^n \rtimes B_n$. Of course, we have isomorphisms: $
\mathcal{F}_{n}\cong {\mathbb Z}^n \rtimes B_n$  and ${\mathcal F}_{d,n}\cong \left( {\mathbb Z}/d{\mathbb Z}\right)^n \rtimes B_n$. Finally, note that the natural projections $C \rightarrow C_d$ and $ B_n \rightarrow S_n$ induce the following commutative diagram:
\[
\xymatrix{ 
&\mathcal{F}_{n} \ar[d] \ar[r] & \mathcal{F}_{d,n}\ar[d] \ar[r] & B_n \ar[d] \ar[r] & 1 \\
&C_{\infty,n} \ar[d] \ar[r] & C_{d,n}\ar[d] \ar[r] & S_n \ar[d] \ar[r] & 1 \\
& 1 &1 &1}
\]

From the above diagram one can define {\it the length function} $l$ on $C_{d,n}$ as the lift of the ordinary length function $l'$ of $S_n$, that is:
\begin{equation}\label{lenfun}
l(t^a s_{i_1}\ldots s_{i_k}) := l'(s_{i_1}\ldots s_{i_k}),
\end{equation}
where $t^a : = t_1^{a_1} \ldots t_n^{a_n}\in C_d^n$.

\begin{remark} \rm
We would like to point out that $C_{d,n}$ and ${\mathcal F}_{d,n}$
appear in the theory of ``fields with one element''. This is a theory 
dreamt by J. Tits in his study of algebraic groups. According to  the 
seminal article of Kapranov and Smirnov \cite{Ka-Sm},  
$\mathrm{GL}_n(\mathbb{F}_1)=S_n$, $\mathrm{GL}_n(\mathbb{F}_1[t])=B_n$, 
$\mathrm{GL}_n(\mathbb{F}_{1^n})=C_{d,n}$ and $\mathrm{GL}_n(\mathbb{F}_{1^n}[t])={\mathcal F}_{d,n}$, where $\mathrm{GL}_n(\mathbb{F}_{1^n})$ 
(resp. $\mathrm{GL}_n(\mathbb{F}_{1^n}[t])$) is in ``some sense'' the 
limit case $q\rightarrow 1$ of $\mathrm{GL}_n(\mathbb{F}_q)$ 
(resp. $\mathrm{GL}_n(\mathbb{F}_q[t])$).
\end{remark}

\subsection{{\it Background material}} We denote by 
 ${\rm H}_n(u)$ the {\it Iwahori--Hecke algebra} associated to $S_n$, that is,  the  $ \mathbb{C}(u)$-algebra with linear basis $\{ h_w \, | \, w\in S_n \}$ and the following rules of multiplication:
 \begin{equation}\label {rulemultH}
h_{s_i}h_w =\left\{\begin{array}{ll}
h_{s_iw} & \text{for } l(s_iw)>l(w) \\
 u h_{s_iw} + (u-1)h_w & \text{for } l(s_iw)<l(w) 
\end{array}\right.  .
\end{equation}

Set $h_i := h_{s_i} $. Then  
  ${\rm H}_n(u)$ is presented by  
  $h_1, \ldots , h_{n-1}$ subject to the following relations:
 \begin{align}
h_i h_j &= h_j h_i \quad \text{for all} \quad |i-j| >1 \label{He1}\\
h_i h_j h_i &= h_j h_i h_j \quad \text{for all} \quad \vert i - j \vert =1 \label{He2}\\
h_i^2 &=  u + (u-1) h_i \label{He3}.
\end{align}

\begin{definition}\label{tlalgebra}\rm
The {\it Temperley--Lieb algebra} ${\rm TL}_n(u)$ can be defined  as the quotient of the algebra ${\rm H}_n(u)$ over the two-sided ideal generated by the {\it Steinberg  elements} $h_{i,j}$:
\begin{equation}\label{idealrel}
h_{i,j}: = \sum_{w \in \langle s_i , s_j \rangle} h_w,\quad \text{for all} \quad \vert i - j \vert =1 .
\end{equation}
\end{definition}
Consequently, the algebra ${\rm TL}_n(u)$ can be thus presented by  $ h_1, \ldots , h_{n-1}$ subject to relations \eqref{He1}--\eqref{He3} and the following relations:
\[
1+ h_i + h_j +h_i h_j + h_j h_i + h_i h_j h_i = 0 \quad \text{for all}\quad |i-j|=1 .
\]
The defining ideal of the algebra ${\rm TL}_n(u)$ is principal and it is generated by the element $h_{1,2}$. Furthermore, using the transformation:
\begin{equation}\label{transfor}
f_i := \frac{1}{u+1}(h_i + 1) ,
\end{equation}
the algebra ${\rm TL}_n(u)$ can be presented by the non-invertible generators $ f_1, \ldots , f_{n-1} $ subject to the following relations:
\[
\begin{array}{ccl}
f_{i}^{2} & = &  f_i \\
f_if_{j}f_i &  = &  \delta f_i,\quad \text{for all} \quad |i-j| =1\\
f_if_j &  = & f_j f_i, \quad \text{for all} \quad |i-j| >1 ,
\end{array}
\]
where $ \delta^{-1} = 2+ u +u^{-1}$ \cite{jo}.
\smallbreak

In \cite{Homfly, jo} Ocneanu constructed a unique Markov trace on ${\rm H}_n (u)$. More precisely, we have the following theorem.
\begin{theorem}[Ocneanu]\label{ocn} Let  $\zeta$ be an indeterminate. There exists a linear trace $\tau$ on $\cup_{n=1}^{\infty} {\rm H}_n(u)$ uniquely defined by the inductive rules:
\begin{enumerate}
\item $\tau (a b) = \tau (ba), \quad a,b \in {\rm H}_n(u)$
\item $\tau (1) = 1$
\item $\tau (a h_n ) = \zeta \,  \tau (a), \quad a\in {\rm H}_n(u)$ \qquad {(Markov property)}.
\end{enumerate}
\end{theorem}

The Ocneanu trace $\tau$ passes through to ${\rm TL}_n(u)$ for specific values of $\zeta$. Indeed, as it turned out \cite{jo}, to factorize $\tau$ to the Temperley--Lieb algebra, we only need the fact that $\tau$ annihilates the expression of Eq. \ref{idealrel}. So, in \cite{jo} it is proved that  
$\tau$ passes to the Temperley--Lieb  algebra if and only if:
\begin{equation}\label{jonval}
\zeta = - \frac{1}{u+1} \quad \mbox{or} \quad \zeta=-1.
\end{equation}

\subsection{{\it The Yokonuma--Hecke algebra}}\label{yhsec}

The Yokonuma--Hecke algebra of type $A$,  denoted by ${\rm Y}_{d,n}(u)$ \cite{yo}, can be defined by generators and relations \cite{ju} and can be regarded as a quotient of  $ \mathbb{C}(u){\mathcal F}_{d,n}$ over the two-sided ideal that is generated by the elements:
\[ \sigma_i^2 - (u-1)e_i - (u-1)e_i \sigma_i -1 , \]
where   $e_{i}$ is the idempotent defined by:
\begin{equation}\label{edi}
e_{i} := \frac{1}{d} \sum_{s=0}^{d-1}t_i^s t_{i+1}^{d-s},
\qquad i=1,\ldots , n-1.
\end{equation}
Equivalently, one can define ${\rm Y}_{d,n}(u)$ as follows:

\begin{definition}\label{yhdefpres} \rm
The {\it Yokonuma--Hecke algebra} ${\rm Y}_{d,n}(u)$ is the algebra  presented by generators $g_1, \ldots, g_{n-1}, t_1 , \ldots , t_n$  subject to the following relations:
\begin{align}
g_ig_j &= g_jg_i \quad \text{\rm for all}\quad  \vert i-j \vert > 1 \label{YH1}\\
 g_{i+1}g_ig_{i+1}&= g_ig_{i+1}g_i \label{YH2}\\
t_it_j &= t_jt_i  \quad \text{\rm for all} \quad  i,j \label{YH4}\\
t_i^{d}  &= 1\quad  \text{\rm for all }\quad  i \label{YH3}\\
g_i t_i &=  t_{i+1}g_i \label{YH5}\\
g_it_{i+1} &= t_i g_i \label{YH6}\\
g_it_j &= t_j g_i \quad \text{\rm for } j \neq i, \, i+1 \label{YH7}\\
g_i^2 &= 1 + (u-1)e_i+ (u-1)e_ig_i \label{quadratic} .
\end{align}

\end{definition}
\noindent
Note that for $d=1$ the quadratic relation \eqref{quadratic} becomes:
\[
g_i^2 = (u-1)g_i + u .
\]
So, the Yokonuma--Hecke  ${\rm Y}_{1,n}(u)$ coincides with the Iwahori--Hecke algebra.\\

The algebra ${\rm Y}_{d,n}(u)$ can also be regarded as a $u$-deformation of the group algebra $\mathbb{C}C_{d,n}$. Indeed, if $w \in S_n$ is a reduced word in $S_n$ with $w = s_{i_1}\ldots s_{i_k}$ then the expression $g_w = g_{s_{i_1}} \ldots g_{s_{i_r}} \in {\rm Y}_{d,n}(u)$ is well-defined since the generators $g_i := g_{s_i}$ satisfy the same braiding relations as the generators of $S_n$ \cite{ma}. We have the 
following multiplication rule in ${\rm Y}_{d,n}(u)$ (see \cite[Proposition 2.14]{jusur}):
\begin{equation}\label{rule}
g_{s_i}g_w =\left\{\begin{array}{ll}
g_{s_iw} & \text{for } l(s_iw)>l(w) \\
 g_{s_iw}+(u-1)e_ig_{s_iw} + (u-1)e_ig_w & \text{for } l(s_iw)<l(w) .
\end{array}\right.
\end{equation}
Note also that the generators $g_{t_i}$ correspond to $t_i$ and so, using Eq.~\ref{lenfun}, we have that: $g_{t_iw} = g_{t_i}g_w =  t_ig_w$.\\

The definition of the idempotents $e_i$ can be generalized in the following way. For any indices $i,j$  we define the following elements in ${\rm Y}_{d,n}(u)$:
\begin{equation}\label{eij}
e_{i,j} := \frac{1}{d} \sum_{s=0}^{d-1}t_i^s t_{j}^{d-s}.
\end{equation}
We also define, for any $0\leq m \leq d-1$, {\it the shift of $e_i$ by $m$}:
\begin{equation}\label{eijm}
e_i^{(m)} := \frac{1}{d} \sum_{s=0}^{d-1} t_i^{m+s} t_{i+1}^{d-s}.
\end{equation}
Notice that $e_i = e_{i, i+1} = e_i^{(0)}$. Notice also that $e_i^{(m)} = t_i^m e_i = t_{i+1}^m e_i$. Then  one deduces easily that:
\begin{align}\label{eimei+1=}
\begin{array}{lcr}
e_i^{(m)}e_{i+1} & = &  e_ie_{i+1}^{(m)}\\
t_1^at_2^bt_3^c e_1 e_2 & = & e_1^{(a+b+c)} e_2 ,
\end{array}
\end{align}
for all $0\leq m, a, b, c\leq d-1$.

The following lemma collects some of the relations among the $e_i$'s, the $t_j$'s and the $g_i$'s. These relations  will be used in the paper.

\begin{lemma}[{\cite[Lemma~1]{gojukola}}]\label{eipropop}
For the idempotents $e_i$  and for $1 \leq i , \,  j \leq n-1$ the following relations hold:
\[
\begin{array}{rcl}
t_j e_i &=& e_i t_j\\
e_{i+1}g_i &=& g_i e_{i,i+2}\\
e_ig_j &=& g_j e_i, \quad \mbox{for } j \neq i-1, i+1\\
e_jg_ig_j &=& g_ig_je_i \quad \mbox{for } |i-j|=1\\
e_ie_{i+1}&=& e_i e_{i, i+2}\\
e_{i}e_{i+1} &=& e_{i,i+2}e_{i+1}.
\end{array}
\]
\end{lemma}

A word in the defining generators of the  algebra will be called a {\it monomial}.  Notice that using relations (\ref{YH5}) and (\ref{YH6}) one can write any monomial $\mathfrak{m}$ in $  {\rm Y}_{d,n}(u)$ in the following form:
\[
\mathfrak{m}= t_1^{a_1}\ldots t_n^{a_n} \mathfrak{m}^\prime ,
\] 
where $\mathfrak{m}^\prime=g_{i_1}\ldots g_{i_n}$. We then say that every monomial in ${\rm Y}_{d,n}(u)$ has the {\it splitting property}, which is in fact inherited from the framed braid group $\mathcal{F}_n$. That is, one can separate the {\it framing part} of $\mathfrak{m}$ (which is the subword in the framing generators $t_j$) from the {\it braiding part} (which is the subword in the braiding generators $g_i$).
\subsection{{\it A Markov trace on ${\rm Y}_{d,n}(u)$}}

 Using the multiplication formulas \eqref{rule}, the second author proved in \cite{ju} that ${\rm Y}_{d,n}(u)$ has the following  standard linear basis:
\begin{equation}\label{standbas}
\{t_1^{a_1}\ldots t_n^{a_n} g_w \, | \,a_i \in C_d,\,  w \in S_n \}.
\end{equation}

\smallbreak

This above linear basis led naturally to the following inductive basis for the Yokonuma--Hecke algebra, which we will use in the proof of the main theorem (Theorem~\ref{akthmgen}). 
\begin{proposition}[{\cite[Proposition 8]{ju}}]\label{inductiveYH}
Every element in ${\rm Y}_{d,n+1}(u)$ is a unique linear combination of words, each of one of the following types:
\[
\mathfrak{m}_{n}g_n g_{n-1}\ldots g_i t_i^k \quad \mbox{or} \quad  \mathfrak{m}_{n}t_{n+1}^k ,
\]
where $0\leq k \leq d-1$ and $\mathfrak{m}_n$ is  a word in the inductive basis of  ${\rm Y}_{d,n}(u)$.
\end{proposition}

Employing the above inductive basis, the second author proved that ${\rm Y}_{d, n}(u)$ supports a unique Markov trace. We have the following theorem:
\begin{theorem}[{\cite[Theorem 12]{ju}}]\label{Juyutrace}
For indeterminates $z$, $x_1$, $\ldots, x_{d-1}$ there exists a unique linear Markov trace ${\rm tr}$:
\[
 {\rm tr}:  \cup_{n=1}^{\infty}{\rm Y}_{d,n}(u) \longrightarrow    \mathbb{C}(u)[z, x_1, \ldots, x_{d-1}],
\]
 defined inductively on $n$ by the following rules:
\[
\begin{array}{rcll}
{\rm tr}(ab) & = & {\rm tr}(ba)  \qquad &  \\
{\rm tr}(1) & = & 1 & \\
{\rm tr}(ag_n) & = & z\, {\rm tr}(a) \qquad &  (\text{Markov  property} )\\
{\rm tr}(at_{n+1}^s) & = & x_s {\rm tr}(a)\qquad  & (  s = 1, \ldots , d-1) ,
\end{array}
\]
where $a,b \in {\rm Y}_{d,n}(u)$.
\end{theorem}

Using the trace rules of Theorem~\ref{Juyutrace} and setting  $x_0:=1$, we deduce that ${\rm tr}(e_i)$ takes the same value for all $i$, and this value is denoted by $E$:
\[
E := {\rm tr}(e_i)= \frac{1}{d}\sum_{s=0}^{d-1}x_{s}x_{d-s}.
\]
Moreover, we also define {\it the shift by $m$ of $E$}, where $0\leq m \leq d-1$, by:
\[
E^{(m)} :={\rm tr}(e_i^{(m)})= \frac{1}{d}\sum_{s=0}^{d-1}x_{m+s}x_{d-s} .
\]
Notice that  $E = E^{(0)}$.

\section{Fourier transform and the ${\rm E}$--system}\label{sectioninv} 
 An important  tool in the  proof of the main theorem are some classical identities of harmonic analysis on  the group of integers modulo $d$.  More precisely, we will use identities linking the convolution product and the product by coordinates  through the Fourier transform. 
These tools were also used in solving the so-called  ${\rm E}$--system, see \cite[Appendix]{jula}. Thus, in this section we  shall give  some notations and recall some well-known and useful facts of the Fourier transform along with some facts for the ${\rm E}$--system. 
\subsection{{\it Computations in $\mathbb{C}C_d$}}
Recall that $C_d$ is the cyclic group of order $d$, generated by  $t$. The {\it product by coordinates} in $\mathbb{C}C_d$ is defined  by the formula: 
\[
\left( \sum_{r=0}^{d-1} a_r t^r\right) \cdot 
\left( \sum_{s=0}^{d-1} b_s t^s\right)=
 \sum_{i=0}^{d-1} a_ib_i t^i
\]
and the  {\it convolution product} is defined by the formula:
\begin{equation} \label{AA1}
\left( \sum_{r=0}^{d-1} a_r t^r\right) \ast  
\left( \sum_{s=0}^{d-1} b_s t^s\right) = 
\sum_{r=0}^{d-1} \left( \sum_{s=0}^{d-1} a_s b_{r-s} \right) t^r .
\end{equation}
In order to define the Fourier transform on $C_d$ we need to introduce the following elements:

\[
\mathbf{i}_a:=\sum_{s=0}^{d-1} \chi_a( t^s) t^s\qquad (a\in {\mathbb Z}/d{\mathbb Z}),
\]
where the  $\chi_k$'s denote the characters of the group $C_d$, namely: 
\begin{equation}\label{charact}
\chi_k( t^m) = \cos\frac{2\pi km}{d} + i \sin \frac{2 \pi km}{d} \qquad (k, m\in {\mathbb{Z}/d\mathbb{Z}} ). 
\end{equation}
One can verify that: 
\[
\mathbf{i}_a \ast \mathbf{i}_b=
\begin{cases}
d\, \mathbf{i}_a & \mbox{if } a=b\\
0 & \mbox{if } a\neq b .
\end{cases}
\]
On the other hand, we shall denote by $\delta_a$ the element $t^a$ of the canonical  linear basis of $\mathbb{C}C_d$. It is clear 
that:
\[
\delta_a \cdot \delta_b=
\begin{cases}
\delta_a & \mbox{if } a=b\\
0 & \mbox{if } a\neq b .
\end{cases}
\]
The {\it Fourier transform} is the linear automorphism on $\mathbb{C}C_d$, defined by: 
\begin{equation}\label{fourtrans}
y:=  \sum_{r=0 }^{d-1}a_rt^r \mapsto \widehat{y}:= \sum_{s=0}^{d-1}(y\ast \mathbf{i}_s)(0)t^s  ,
\end{equation}
where $(y\ast \mathbf{i}_s)(0)$ denote the coefficient of $\delta_0$ in the convolution $y\ast \mathbf{i}_s$.

The next proposition gathers  the most important properties of the Fourier transform used in the paper.
\begin{proposition}[{\cite[Chapter~2]{te}}]\label{propietrans}
For any $y$ and $y^{\prime}$ in $\mathbb{C}C_d$, we have:
\begin{enumerate}
\item $\widehat{y \ast y^{\prime}}=\widehat{y}\cdot \widehat{y^{\prime}}$
\item $\widehat{y\cdot y^{\prime}}= d^{-1}\widehat{y} \ast \widehat{y^{\prime}}$
\item $\widehat{\delta}_a=\mathbf{i}_{-a}$
\item $\widehat{\mathbf{i}}_a=d \delta_a$
\item If $ y = \sum_{r=0}^{d-1}a_rt^r$, then 
$
\widehat{\widehat{y}}=  d \sum_{r=0}^{d-1}a_{-r}t^r .
$
\end{enumerate}
\end{proposition}

Finally, we note that  the  elements in the group algebra $\mathbb{C}C_d$ can also
be identified to the set of functions $f:C_d\rightarrow \mathbb{C}$,
where the identification is as follows:
\begin{equation} \label{coresp1}
(f:C_d\rightarrow \mathbb{C})   \longleftrightarrow \sum_{k=0}^{d-1} f(t^k) t^k\in \mathbb{C}C_d.
\end{equation}
Some times we shall use this identification, since it makes some computations easier.
\subsection{{\it The ${\rm E}$--system and its solutions}}\label{esyssec}
The ${\rm E}$--system is a non-linear system of equations that was introduced in order to find the  necessary and sufficient conditions that need to be applied on the parameters $x_i$  of  $\rm tr$ so that the definition of link invariants from the Yokonuma--Hecke algebra would be possible \cite{jula}.

\begin{definition}[{\cite[Definition 11]{jula}}]
 \rm
We say that the $(d-1)$-tuple  of complex numbers $(\rm{x}_1, \ldots , \rm{x}_{d-1})$ satisfies the {\it ${\rm E}$--condition} if
$\rm{x}_1,\ldots , \rm{x}_{d-1}$ satisfy the following system of non-linear equations in $\mathbb{C}$, the {\it ${\rm E}$--system}:
\begin{equation}\label{Esystem}
E^{(m)}  =   {\rm x}_{m}E \qquad (1\leq m \leq d-1).
\end{equation}
\end{definition}

In \cite[Appendix]{jula} the full set of solutions of the ${\rm E}$--system is given by G\'{e}rardin using some tools of harmonic analysis on finite groups. More precisely, he interpreted the solution 
{$({\rm x}_1,\ldots,{\rm x}_d)$} 
of the ${\rm E}$--system,   
 as a certain complex function $x_{_D} : C_d \rightarrow \mathbb{C}$. The solution is 
parametrized by a non-empty subset $D$ of $C_d^*$, where 
$C_d^*$ denotes the dual group of $C_d$, i.e. the space of characters of $C_d$. 
Since $C_d \cong C_d^* \cong \mathbb{Z}/d \mathbb{Z}$, by small abuse of notation, we will consider  $D$  as  a subset of $\mathbb Z/d\mathbb Z$. Recall that the characters $\chi_k$ of $C_d$ are given by $t^a \mapsto \chi_k(t^a)$, where $k$ runs over $\mathbb{Z}/d\mathbb{Z}$, see Eq.~\ref{charact}.

The dependence of $x_D$ on $D$ is given by the following equation of functions:
 \begin{equation}\label{solEsys}
x_{_D}={\frac{1}{ \vert D\vert}}\sum_{k\in D}\chi_k.
\end{equation}

Notice that the function $x_{_D}$ can be also seen as an  element in $\mathbb{C}C_d$, namely:
\begin{equation}\label{solECd}
x_{_D} = \sum_{j=0}^{d-1} {\rm x}_j t^j ,
\end{equation}
where  ${\rm x}_j =x_{_D}(t^j)= \frac{1}{\vert D\vert }\sum_{k\in D}\chi_{k}(t^j)$.

A simple computation shows that the convolution products, where $x$ is an element in the group algebra $\mathbb{C}C_d$, are given by:
\begin{equation}\label{x*x}
x \ast x = d \sum_{k=0}^{d-1}  \mathrm{tr} (e_i^{(k)})t^k =d \sum_{k=0}^{d-1}  E^{(k)}t^k, 
\qquad
x\ast x \ast x =d^2 \sum_{k=0}^{d-1} \mathrm{tr}(e_1^{(k)}e_2) t^k,
\end{equation}
see also  \cite[Lemma 2]{gojukola}

\begin{remark} \rm
It is worth noting that the formula for the solutions of the ${\rm E}$--system can be interpreted as a generalization of the Ramanujan sum. Indeed, by taking the subset $P$ of  $C_d$ consisting  
of the numbers coprime to $d$, then the solution parametrized by $P$ is, up to the factor $\vert P\vert$,  the Ramanujan sum $c_d(k)$ (see \cite{ra}).
\end{remark}

We  finish this section with a theorem which yields the main connection among the solutions of the ${\rm E}$--system and the trace ${\rm tr}$.  

\begin{theorem}[{\cite[Theorem 7]{jula}}]\label{trmultip}
If the trace parameters  $(x_1, \ldots , x_{d-1} )$ satisfy the  ${\rm E}$--condition,  then 
\[
{\rm tr}(\alpha e_n) = {\rm tr}(\alpha) {\rm tr}(e_n)\qquad (a \in {\rm Y}_{d,n}(u)).
\]
\end{theorem}

\section{A Framization of the Temperley--Lieb algebra} \label{motivation}
In this section we explore three quotients of the Yokonuma--Hecke algebra, ${\rm YTL}_{d,n}(u)$, ${\rm FTL}_{d,n}(u)$ and ${\rm CTL}_{d,n}(u)$, as potential candidates  for the framization of the Temperley--Lieb algebra and we select one of them, namely ${\rm FTL}_{d,n}(u)$, as the most appropriate in view of our topological aims.
\subsection{{\it The three potential candidates}} As discussed   during  the Introduction, the Yokonuma--Hecke algebra can be interpreted as the  framization of the Iwahori--Hecke algebra, which is a knot algebra. Thus a natural question arises, the definition of  a framization for the knot algebra Temperley--Lieb. Considering the fact that the Temperley--Lieb algebra can be defined as a quotient of the Iwahori--Hecke algebra,  it is natural to try and define a framization of the Temperley--Lieb algebra as a quotient of the Yokonuma--Hecke algebra. Recall now  that the defining ideal of the Temperley--Lieb algebra (Definition \ref{tlalgebra}) is generated by the Steinberg elements which are related to the subgroups $\langle  s_i, s_{i+1}\rangle$ of $S_n$, for all $i$. These subgroups can be also regarded as subgroups of $C_{d,n}$. Therefore, using the multiplication rule of  Eq. \ref{rule} we are able to define the analogous Steinberg elements $g_{i,i+1}$ in ${\rm Y}_{d,n}(u)$, 
\begin{equation}\label{ytldefid}
g_{i,i+1} := \sum_{w\in \langle  s_i, s_{i+1}\rangle}g_w \qquad \text{for all } i  .
\end{equation}
In \cite[Definition 2]{gojukola} we defined a potential candidate for the framization of the Temperley--Lieb algebra, the  {\it Yokonuma--Temperley--Lieb algebra}, denoted by ${\rm YTL}_n(u)$ which is defined as the quotient of  ${\rm Y}_{d,n}(u)$ over the two-sided ideal generated by the  $g_{i,i+1}$'s for all $i$.  It is not difficult to show that this ideal is in fact principal and it is generated by the element $g_{1,2}$. Moreover, the necessary and sufficient conditions  for the trace ${\rm tr}$ to pass through to ${\rm YTL}_n(u)$ were studied \cite[Theorem 6]{gojukola}. Unfortunately, these conditions turn out to be too strong. Namely, the trace parameters $x_i$ must be $d^{th}$ roots of unity, giving rise to obvious, special solutions of the ${\rm E}$--system, which imply topologically loss of the framing information. Moreover, if we restrict to the case of classical links, by representing the Artin braid group $B_n$ in ${\rm Y}_{d,n}(u)$, considering the $t_i$'s as formal generators, and then taking the quotient over the ideal that is generated by the $g_{i,i+1}$'s \cite[Section 5]{gojukola}, and using the results of \cite{ChLa}, the derived classical link invariants for the algebras ${\rm YTL}_{d,n}(u)$ coincide with the classical Jones polynomial.
For the above reasons, ${\rm YTL}_{d,n}(u)$ is discarded as framization of ${\rm TL}_n(u)$. Finally, we note that the representation theory of this algebra  has been studied extensively in \cite{ChPou}. 

\smallbreak

Given the fact that ${\rm Y}_{d,n}(u)$ can be considered as a $u$-deformation of $\mathbb{C}C_{d,n}$ (recall the discussion in Section~\ref{yhsec}), it is natural to consider subgroups of $C_{d,n}$ that involve in their generating set the framing generators of the $i$-th and $j$-th strands along with $\langle s_i, s_j \rangle$. As a first attempt, we consider the following subgroups of $C_{d,n}$:
\[
C^{i}_{d,n} := \langle t_i, t_{i+1}, t_{i+2}\rangle \rtimes \langle  s_i, s_{i+1}\rangle\quad \text{for all } i. 
\] 
Notice that these  subgroups are isomorphic to the group $C_{d,3}$, in analogy to the classical case of ${\rm TL}_n(u)$. We define now the elements $c_{i,i+1}$ in ${\rm Y}_{d,n}(u)$ as follows:
\begin{equation}\label{defcij}
 c_{i,i+1} =  \sum_{c\in C_{d,n}^{i}}g_{c}  .
\end{equation}
We then have the following definition:
\begin{definition} \rm
For $n\geq 3$, we define the algebra ${\rm CTL}_{d,n}(u)$ as the quotient of  the  algebra ${\rm Y}_{d,n}(u)$ by the two-sided ideal generated by the  $c_{i,i+1}$'s, for all $i$.
 We shall call ${\rm CTL}_{d,n}(u)$  the {\it Complex Reflection Temperley--Lieb algebra}.
\end{definition}
\begin{remark} \rm
The denomination Complex Reflection Temperley--Lieb algebra has to do with the fact that the underlying group of ${\rm CTL}_{d,n}(u)$ is isomorphic to the complex reflection group $G(d,1,3)$.
\end{remark}
\smallbreak
As it will be shown in Theorem  \ref{ctlthm}, the necessary and sufficient conditions such that $\rm tr$ passes to ${\rm CTL}_{d,n}(u)$ are, contrary to the case of ${\rm YTL}_{d,n}(u)$, too relaxed, especially on the trace parameters $x_i$. So, in order to define link invariants from the algebras ${\rm CTL}_{d,n}(u)$, the ${\rm E}$--condition must be imposed on the $x_i$'s. 
\smallbreak

This indicates that the desired framization of the Temperley--Lieb algebra for our topological purposes could be an intermediate algebra between these two. We achieve this, by using for the defining ideal an intermediate subgroup of $C_d^n$ that lies between  $\langle  s_i, s_{i+1}\rangle $ and $C^{i}_{d,n}$. Indeed, we consider the following subgroups of $C_{d,n}$, 
\[
H^{i}_{d,n} := \langle t_i t_{i+1}^{-1}, t_{i+1}t_{i+2}^{-1} \rangle \rtimes \langle s_i, s_{i+1} \rangle \quad \text{for all } i.
\]
We now introduce the following elements:
\[
r_{i,i+1} :=   \sum_{x\in H^{i}_{d,n}} g_x\quad \text{for all } i. 
\]
\begin{definition}\label{ftldefine}\rm
For $n \geq 3$, the {\it Framization of the Temperley--Lieb algebra}, denoted  ${\rm FTL}_{d,n}(u)$, is defined as the quotient ${\rm Y}_{d,n}(u)$ over the two-sided ideal generated by the elements $r_{i,i+1}$, for all  $i$.
\end{definition}

\smallbreak
 
\begin{remark}\label{collapse}\rm 
Notice that when $d=1$, the Yokonuma--Hecke algebra coincides with the Iwahori--Hecke algebra, hence it follows that  ${\rm YTL}_{1,n}(u)$ also coincides with  ${\rm TL}_n(u)$. Moreover, in this case 
the subgroups $H^{i}_{d,n}$ and $C^{i}_{d,n}$ also collapse to $  \langle s_i, s_{i+1} \rangle$, which is isomorphic to $S_3$. Hence,   ${\rm FTL}_{1,n}(u)$ and ${\rm  CTL}_{1, n}(u)$ coincide with  ${\rm TL}_n(u)$ too.
\end{remark}    

\subsection{{\it Relating the three quotient algebras}}
We shall now show how the algebras defined above are related. Notice that the defining ideal for each quotient algebra mentioned above is generated by  sums of elements $g_x$, where $x$ belongs to the underlying group of each ideal. More precisely, the underlying group of the defining ideal of ${\rm YTL}_{d,n}(u)$ is $S_3$ of ${\rm FTL}_{d,n}(u)$ is $H^{i}_{d,n}$ and of ${\rm CTL}_{d,n}(u)$ is $C^{i}_{d,n}$. We have the following inclusion of groups : $ S_3 \leq H_{d,n}^i \leq C_{d,n}^i$. We will show that this implies the following inclusions of ideals:
\begin{equation}\label{idealincl}
 \langle c_{i,i+1} \rangle \lhd \langle r_{i,i+1} \rangle \lhd \langle g_{i,i+1} \rangle .
\end{equation}
The second inclusion of the ideals, $\langle r_{i,i+1} \rangle \lhd \langle g_{i,i+1} \rangle$, is clear. Indeed, every $x$ in $H^{i}_{d,n}$ can be written in the form:
\[
x=t_i^at_{i+1}^{-a}t_{i+1}^{b}t_{i+2}^{-b}\,w = t_i^at_{i+1}^{b-a}t_{i+2}^{-b}\,w, \quad \mbox{where } w\in  S_3.
\]
Therefore, from the multiplication rule of Eq. \ref{rule}, we have  that $g_x = t_i^at_{i+1}^{b-a}t_{i+2}^{-b}g_w$. Thus we can rewrite the elements $r_{i, i+1}$ in the following form:

\[
r_{i,i+1} = \sum_{\stackrel{a, b = 0}{w\in S_3}}^{d-1}t_i^at_{i+1}^{b-a}t_{i+2}^{-b}\, g_w 
= \left(\sum_{a, b = 0}^{d-1}t_i^at_{i+1}^{b-a}t_{i+2}^{-b}\right) 
\left(\sum_{w \in S_3} g_w \right),
\]
hence
\begin{equation}\label{rij=}
r_{i,i+1} = d^2 e_ie_{i+1} g_{i,i+1} .
\end{equation}
We shall proceed now with the proof of the first inclusion of ideals. We observe that:
\begin{equation}\label{semdirprod}
 C^{i}_{d,n} = H^{i}_{d,n} \rtimes C_d .
\end{equation} 
Indeed, let $x = t_i^a t_{i+1}^b t_{i+2}^c \, w$ an element in $ C^{i}_{d,n}$, where $w\in S_3 $, and let $\phi$ be the following homorphism:
 \begin{align*}
 \phi: C^{i}_{d,n}& \to \langle t_i \rangle \cong  C_d\\
 x &\mapsto t_i^{a+b+c}.
\end{align*}
Observe that $\ker\phi = H^{i}_{d,n}$, so $\phi\big |_{H^{i}_{d,n}}=  {\rm id}_{C_d}$, which implies   Eq.~\ref{semdirprod}. 
  Therefore, for the element  $x \in C^{i}_{d,n}$ we have a unique decomposition $x= t_i^k y$, where $0\leq k\leq  d-1$ and  $y \in  H^{i}_{d,n}$. This decomposition of the elements of $C^{i}_{d,n}$ together with   the multiplication rule in Eq. \ref{rule}, implies $g_x = t_i^k g_y$. This allows us to write the elements $c_{i,i+1}$ of Eq.~\ref{defcij} in the following equivalent form:
\[
c_{i,i+1}  = \sum_{ \stackrel{0\leq k\leq  d-1}{y\in H^{i} _{d,n}}}t_i^{k}  g_y ,
\]
hence:
\begin{equation}\label{cij=}
c_{i,i+1}  = \left( \sum_{ k= 0 }^{d-1}t_i^{k} \right) r_{i,i+1} .
\end{equation}

Equation~\ref{cij=}  implies that ${\rm CTL}_{d,n}(u)$ projects onto ${\rm FTL}_{d,n}(u)$ while Eq.~\ref{rij=}  implies that ${\rm FTL}_{d,n}(u)$ projects onto ${\rm YTL}_{d,n}(u)$. 
We have thus proved the following:
\begin{proposition}\label{idealinclprop}
The inclusions of ideals of Eq.~\ref{idealincl} yield the following natural commutative diagram of epimorphisms:
\[
\xymatrix{ 
&{\rm Y}_{d,n}(u) \ar[d] \ar[r] & {\rm CTL}_{d,n}(u)\ar[d] \ar[r] & {\rm FTL}_{d,n}(u) \ar[dl] \ar[r] & {\rm YTL}_{d,n}(u) \ar@/^/[dll]\\
& {\rm H}_n(u) \ar[r] & {\rm TL}_n (u) & &}
\]
where the non-horizontal arrows are defined by mapping the framing generators to $1$.
\end{proposition}

\subsection{{\it Principality of the ideals}}
 
It is known that the defining ideal of the Temperley--Lieb algebra is principal \cite{jo}. We are going now to prove that the defining ideals of ${\rm FTL}_{d,n}(u)$ and ${\rm CTL}_{d,n}(u)$ respectively are principal ideals too. The method used in the proof is standard \cite{jo} but for self-containedness of the paper we will sketch the basic ideas. We  start with a technical lemma. 

\begin{lemma}\label{lemmaforprop}
The following hold in ${\rm Y}_{d,n}(u)$ for all $i =1, \ldots , n-2$,  $j = 1 , \ldots , n $ and $0\leq a, b, c \leq d-1$:

\[
\begin{array}{crcl}
{\it (1)} &t_{j} &=& ( g_1 \ldots g_{n-1})^{j-1}\, t_1 \, (g_1\ldots g_{n-1})^{-(j-1)}\\
{\it (2)} &g_{i} &=&( g_1 \ldots g_{n-1})^{i-1} \, g_1 \, (g_1\ldots g_{n-1})^{-(i-1)}\\
{\it (3)} &t_i^a t_{i+1}^b t_{i+2}^c &=& (g_1 \ldots g_{n-1})^{i-1}\, t_1^a t_2^b t_3^c \,(g_1 \ldots g_{n-1})^{-(i-1)}\\
{\it (4)} &t_i^a t_{i+1}^b t_{i+2}^c g_i &=& (g_1 \ldots g_{n-1})^{i-1} \,t_1^a t_2^b t_3^c g_1\,(g_1 \ldots g_{n-1})^{-(i-1)}\\
{\it (5)} &t_i^a t_{i+1}^b t_{i+2}^c g_{i+1}&=&(g_1 \ldots g_{n-1})^{i-1} \,t_1^a t_2^b t_3^c g_2\, (g_1 \ldots g_{n-1})^{-(i-1)}\\
{\it (6)} &t_i^a t_{i+1}^b t_{i+2}^c g_{i} g_{i+1} &=& (g_1 \ldots g_{n-1})^{i-1} \,t_1^a t_2^b t_3^c g_1g_2\,(g_1 \ldots g_{n-1})^{-(i-1)}\\
{\it (7)} &t_i^a t_{i+1}^b t_{i+2}^c g_{i+1}g_i &=& (g_1 \ldots g_{n-1})^{i-1}\, t_1^a t_2^b t_3^c g_2 g_1\,(g_1 \ldots g_{n-1})^{-(i-1)}\\
{\it (8)} & t_i^a t_{i+1}^b t_{i+2}^c g_ig_{i+1} g_i &=& (g_1 \ldots g_{n-1})^{i-1}\, t_1^a t_2^b t_3^c g_1 g_2 g_1 \,(g_1 \ldots g_{n-1})^{-(i-1)}.
\end{array} 
\]
\end{lemma}

\begin{proof}
Statement $(1)$ is proved by application of Eqs.~\ref{YH5}--\ref{YH7} and induction on $j$. The proof of the second statement is standard in the literature; it follows from the braid relations \eqref{YH1} and \eqref{YH2} and induction on $i$. The other statements of the Lemma are proved by repeated applications of statements $(1)$ and $(2)$.
\end{proof}

\begin{lemma}\label{prop2} 
The defining ideal of ${\rm FTL}_{d,n}(u)$  is generated   by any  single element $r_{i,i+1}$.
\end{lemma} 
\begin{proof}
It is enough to prove that $r_{i, i+1}= (g_1 \ldots g_{n-1})^{(i-1)}r_{1,2} (g_1 \ldots g_{n-1})^{-(i-1)}$. Indeed, expanding $r_{1,2}$ in the right-hand side of the equality, we have:
\begin{align}
(g_1 \ldots g_{n-1})^{i-1}r_{1,2} (g_1\ldots g_{n-1})^{-(i-1)}
  &=  
\sum_{\substack{a, b =0\\ w \in S_3 }}^{d-1} (g_1 \ldots g_{n-1})^{i-1}\, t_1^a t_2^{b-a} t_{3}^{-b}\, g_w \,(g_1 \ldots g_{n-1})^{-(i-1)} \nonumber \\
&= \sum_{a, b =0}^{d-1} (g_1 \ldots g_{n-1})^{i-1}\, t_1^a t_2^{b-a} t_{3}^{-b}\, \left ( \sum_{w \in S_3} g_w \right ) \,(g_1 \ldots g_{n-1})^{-(i-1)} \nonumber \\
&= r_{i,i+1}, \label{r12principal}\nonumber
\end{align}
Therefore the proof is concluded.
\end{proof}
The following is an immediate corollary of Lemma~\ref{prop2}.
\begin{corollary}\label{ftlpr}
${\rm FTL}_{d,n}(u)$ is the  algebra   generated by $ t_1,\ldots , t_n , g_1 , \ldots , g_{n-1}$ which are subject to the defining relations of ${\rm Y}_{d,n}(u)$ and the relation:
\begin{equation}\label{eiform}
r_{1,2}=0.
\end{equation}
\end{corollary}
Further, an analogous result  (with analogous proofs) holds for the algebra ${\rm CTL}_{d,n}(u)$. So we have the following:
\begin{corollary}\label{ctlpr}
The  defining ideal of ${\rm CTL}_{d,n}(u)$ is generated by any single element $c_{i,i+1}$. Hence ${\rm CTL}_{d,n}(u)$ 
can be presented by  $ t_1,\ldots , t_n , g_1 , \ldots , g_{n-1}$ together with the defining relations of ${\rm Y}_{d,n}(u)$ and the relation:
\begin{equation}\label{eiformC}
c_{1,2}=0.
\end{equation}
\end{corollary}

\subsection{{\it Presentations with non-invertible generators}}
By using the analogous transformation to Eq. \ref{transfor}, we obtain presentations for ${\rm FTL}_{d,n}(u)$ and  ${\rm CTL}_{d,n}(u)$ through non-invertible generators. More precisely, 
set 
\begin{equation}\label{elltog}
\ell_i := \frac{1}{u+1}(g_i+1).
\end{equation}

\begin{proposition}\label{presFTL}
The algebra ${\rm FTL}_{d,n}(u)$ can be presented with generators $ \ell_1,  \ldots , \ell_{n-1}, t_1 , \ldots , t_n$, subject to the following relations:
\begin{align}
\ell_i\ell_j &=\ell_j\ell_i,  \quad \text{for} \quad |i-j| >1\\
\ell_i\ell_{i+1}\ell_i - \frac{(u-1)e_i +1}{(u+1)^2}\,\ell_i &= \ell_{i+ 1}\ell_{i}\ell_{i + 1} - \frac{(u-1)e_{i+1} +1}{(u+1)^2}\,\ell_{i+ 1}   \label{noninveq1}\\
t_i^d &=1, \hspace{.3cm} t_it_j=t_jt_i \\ 
\ell_it_i &= t_{i+1}\ell_i + \frac{1}{u+1}(t_i - t_{i+1}) \\
\ell_i t_{i+1}&= t_i\ell_i +\frac{1}{u+1}(t_{i+1}-t_i) \\
\ell_i t_j& = t_j \ell_i, \quad \text{for} \quad |i-j| >1\\
\ell_i^2 &= \frac{(u-1)e_i+2}{u+1}\,\ell_i, \label{ellquad} \\
e_i e_{i +1}\ell_i\ell_{i +1}\ell_i &=\frac{u}{(u+1)^2}\, e_ie_{i+1}\ell_i \label{noninvrel} .
\end{align}
\end{proposition}

\begin{proof}

It is a straightforward computation to see that relations (\ref{YH1}) -- (\ref{quadratic}) 
are transformed via Eq.~\ref{elltog} into the relations \eqref{noninveq1} -- (\ref{noninvrel}). We will prove here some indicative cases. The rest are proved in  an analogous way.  First we will prove the quadratic relation \eqref{ellquad}. From Eq.~\ref{elltog} we have that:
\[
g_i^2 = \left ((u+1) \ell_i -1  \right )^2,
\]
using Eq.~\ref{quadratic}, this is equivalent to:
\[
1+ (u-1)e_i + (u-1)e_ig_i = (u+1)^2 \ell_i^2 - 2 (u+1) \ell_i +1,
\]
or, via Eq.~\ref{elltog}, equivalently:
\[ 
(u-1)(u+1)e_i \ell_i  = (u+1)^2 \ell_i^2 - 2(u+1)\ell_i,
\]
which leads to :
\[ 
\ell_i^2 = \frac{(u-1) e_i + 2}{u+1} \ell_i.
\]

\vspace{.3cm}
Next we will prove Eq.~\ref{noninveq1}. From Eq.~\ref{elltog} we obtain:
 \begin{align}
 g_ig_{i+1}g_i &= (u+1)\ell_i \ell_{i+1} \ell_i  - (u+1)^2 \ell_i^2 - (u+1)^2 \ell_{i+1} \ell_i + (u+1)\ell_i \nonumber \\
 &\, - (u+1)^2\ell_i \ell_{i+1} + (u+1) \ell_i + (u+1)\ell_{i+1} -1 \label{gigjgi}.\\
 g_{i+1}g_{i}g_{i+1} &= (u+1)\ell_{i+1} \ell_{i} \ell_{i+1}  - (u+1)^2 \ell_{i+1}^2 - (u+1)^2 \ell_{i} \ell_{i+1} + (u+1)\ell_{i+1} \nonumber \\
 &\, - (u+1)^2\ell_{i+1} \ell_{i} + (u+1) \ell_{i+1} + (u+1)\ell_{i} -1. \label{gjgigj}
 \end{align}
Equations \ref{YH2}, \ref{gigjgi}, \ref{gjgigj} and \ref{ellquad} lead us to the desired result:
\[
 \ell_i \ell_{i+1} \ell_i - \frac{(u-1) e_i +1}{(u+1)^2}\, \ell_i = \ell_{i+1} \ell_i \ell_{i+1} - \frac{(u-1)e_{i+1} +1}{(u+1)^2} \,\ell_{i+1}, \quad 1 \leq i \leq n-2.
 \]
\vspace{.3cm}
 Finally, from relations $e_i e_{i+1} g_{i,i+1}=0$ using Eqs.~\ref{ytldefid} and \ref{elltog} we have  for $1 \leq i \leq n-2$ that:
\begin{align*}
0= e_ie_{i+1}g_{i, i + 1}&=e_i e_{i+1} \left (g_{i}g_{i + 1}g_i + g_{i + 1}g_i + g_{i} g_{i + 1} + g_{i +1} + g_i +1 \right)\\
& = e_ie_{i+1} \left ((u+1)^3 \ell_i \ell_{i +1} \ell_i - (u+1)^2 \ell_i^2 + (u+1) \ell_i \right ).
\end{align*}
From Eq.~\ref{ellquad} we have that:
\[
 e_ie_{i+1} \left ( (u+1)^2\ell_i \ell_{i +1} \ell_i \right) = e_i e_{i+1}\big((u-1) e_i +1 \big) \ell_i ,
\]
or equivalently:
\[
e_i e_{i+1} \ell_i \ell_{i +1} \ell_i =\frac{u}{(u+1)^2}\, e_ie_{i+1} \ell_i,
\]
which is Eq.~\ref{noninvrel}.
\end{proof}

\begin{proposition}\label{ctlprop}
The algebra ${\rm CTL}_{d,n}(u)$ can be presented with generators $
 \ell_1,  \ldots , \ell_{n-1}$, $t_1 , \ldots , t_n$, subject to the following relations:
\begin{align*}
\ell_i\ell_j &=\ell_j\ell_i,  \quad \text{for} \quad |i-j| >1\\
\ell_i\ell_{i+1}\ell_i - \frac{(u-1)e_i +1}{(u+1)^2}\,\ell_i &= \ell_{i+ 1}\ell_{i}\ell_{i + 1} - \frac{(u-1)e_{i+1} +1}{(u+1)^2}\,\ell_{i+ 1}   \\
t_i^d &=1, \hspace{.3cm} t_it_j=t_jt_i \\ 
\ell_it_i &= t_{i+1}\ell_i + \frac{1}{u+1}(t_i - t_{i+1}) \\
\ell_i t_{i+1}&= t_i\ell_i +\frac{1}{u+1}(t_{i+1}-t_i) \\
\ell_i t_j& = t_j \ell_i, \quad \text{for} \quad |i-j| >1\\
\ell_i^2 &= \frac{(u-1)e_i+2}{u+1}\,\ell_i  \\
\sum_{k=0}^{d-1}e_i^{(k)} e_{i +1}\ell_i\ell_{i +1}\ell_i &= \sum_{k=0}^{d-1}e_i^{(k)} e_{i+1}\frac{u}{(u+1)^2}\,\ell_i .
\end{align*}
\end{proposition}

\begin{proof}
The proof is a straightforward computation and totally analogous to the proof of Proposition~\ref{presFTL}.
\end{proof}

\begin{remark}\rm
We know that a linear basis of the Temperley--Lieb  algebra can be constructed from the interpretation of  the generators  $\ell_i$ as diagrams. In virtue of Remark \ref{collapse},  then  it is desirable to construct  a basis of ${\rm FTL}_{d,n}(u)$ from the presentation given in Proposition \ref{presFTL}. Unfortunately, we do not have a diagrammatic interpretation for the generators $\ell_i$ yet. 
In a recent result \cite{ChPou} Chlouveraki and Pouchin studied extensively the representation theories of the algebras ${\rm FTL}_{d,n}(u)$ and ${\rm CTL}_{d,n}(u)$. Further, they provided linear bases for both  and they  also computed their dimensions. We will present here the dimensions of both of the algebras ${\rm FTL}_{d,n}(u)$ and ${\rm CTL}_{d,n}(u)$. For this purpose, let ${\rm Comp}_d (n) : =\left \{ \mu= (\mu_1,\, \mu_2,\,  \ldots , \mu_d ) \in \mathbb{N}^{d} \, \vert \, \mu_1 + \mu_2 + \ldots + \mu_d = n \right \} $ and let also $c_k:= \frac{1}{k+1} {2k \choose k}$ be the $k$-th Catalan number. We then have:
\end{remark}

\begin{theorem}[{\cite[Theorems~3.10, 5.5 and Remark 5.6]{ChPou}}]
The dimension of the quotient algebra ${\rm FTL}_{d,n}(u)$ is:
\begin{equation}\label{dimftl}
{\rm dim}_{\mathbb{C}(u)}{\rm FTL}_{d,n}(u) = \sum_{\mu \, \in \, {\rm Comp}_d (n)} \left ( \frac{n!}{\mu_1! \, \mu_2 ! \, \ldots \mu_d !} \right )^2 c_{\mu_1}\, c_{\mu_2}\ \ldots \, c_{\mu_d}.  
\end{equation} 
The dimension of the quotient algebra ${\rm CTL}_{d,n}(u)$ is:
\[
{\rm dim}_{\mathbb{C}(u)}{\rm CTL}_{d,n}(u) = \sum_{\mu\, \in \, {\rm Comp}_d(n)} \left( \frac{n !}{\mu_1 ! \, \mu_2 ! \, \ldots \, \mu_d !} \right)^2 c_{\mu_1}\, \mu_2 ! \, \ldots \, \mu_d! .
\]
\end{theorem}

\subsection{{\it Technical lemmas}}
We finish this section with two technical  lemmas concerning the interaction with the braiding generators $g_1$, $g_2$ of the generators $g_{1,2}$, $r_{1,2}$, $c_{1,2}$ of the three ideals discussed above. Also, these lemmas will be used in the proof of  Theorems~\ref{akthmgen} and \ref{ctlthm}.

\begin{lemma}\label{condeqs}
For the element $g_{1,2}$ we have in ${\rm Y}_{d,n}(u)$ the following:
\[
\begin{array}{crcl}
{\it (1)} &g_1 g_{1,2}& =& [1+ (u-1) e_1]g_{1,2}\\
{\it (2)} &g_2g_{1,2} &= &[1 +(u-1)e_2]g_{1,2}\\
{\it (3)}&g_1 g_2 g_{1,2} &=& [ 1 +(u-1) e_1+(u-1) e_{1,3}+(u-1)^2 e_1e_2 ]g_{1,2}\\
{\it (4)} &g_2g_1g_{1,2} &=& [ 1 +(u-1)\, e_2+(u-1) e_{1,3}+(u-1)^2e_1e_2 ] g_{1,2}\\
{\it (5)}&g_1g_2g_1g_{1,2} &=& [ 1+ (u-1)(e_1+e_2+e_{1,3})+(u-1)^2(u+2)\,e_1e_2 ] g_{1,2}.
\end{array}
\]
\end{lemma}
\begin{proof}
See \cite[Lemma 5]{gojukola}. Cf. \cite[Lemma 7.5 ]{juptl}.
\end{proof}

\begin{lemma}\label{condeqsftl}
For the element $r_{1,2}$ we have in ${\rm Y}_{d,n}(u)$:
\[
\begin{array}{crcl}
{\it (1)} &g_1 r_{1,2} &=& [1+(u-1)e_1]r_{1,2}\\
{\it (2)} & g_2 r_{1,2} &=& [1+(u-1)e_2]r_{1,2}\\
{\it (3)} &g_1g_2 r_{1,2} &= &[ 1 +(u-1) e_1+(u-1) e_{1,3}+(u-1)^2 e_1e_2 ]r_{1,2}\\
{\it (4)} &g_2g_1 r_{1,2} &=& [ 1 +(u-1)\, e_2+(u-1) e_{1,3}+(u-1)^2e_1e_2 ] r_{1,2}\\
{\it (5)} &g_1g_2 g_1r_{1,2} &=& [ 1+ (u-1)(e_1+e_2+e_{1,3})+(u-1)^2(u+2)\,e_1e_2 ] r_{1,2}.
\end{array}
\]
\end{lemma}

\begin{proof}
For proving this lemma we will make extensive use of Lemmas~\ref{condeqs} and \ref{eipropop}. For  statement $(1)$ we have:
\begin{align*}
g_1 r_{1,2} &= g_1 e_1 e_2 g_{1,2} = e_1 e_{1,3} g_1 g_{1,2}\\
&= e_1 e_2 [1 + (u-1) e_1 ] g_{1,2} \\
&= [1 + (u-1) e_1] e_1 e_2 g_{1,2}\\
&= [1+ (u-1) e_1] r_{1,2}.
\end{align*}
In an analogous way we prove statement  (2). For statement  (3) we have that:
\begin{align*}
g_1 g_2 r_{1,2} &= g_1 g_2 e_1 e_2 g_{1,2} = e_2e_{1,3} g_1g_2 g_{1,2}\\
&=e_1 e_2 [ 1 +(u-1) e_1+(u-1) e_{1,3}+(u-1)^2 e_1e_2 ]g_{1,2}\\
&= [ 1 +(u-1) e_1+(u-1) e_{1,3}+(u-1)^2 e_1e_2 ]e_1 e_2g_{1,2}\\
&= [ 1 +(u-1) e_1+(u-1) e_{1,3}+(u-1)^2 e_1e_2 ]r_{1,2}.
\end{align*}
In an analogous way we prove statement (4). Finally, we have for statement (5):
\begin{align*}
g_1 g_2 g_1 r_{1,2} &= g_1 g_2 g_1 e_1 e_2 g_{1,2} \\
&= e_1 e_2 g_1 g_2 g_1 g_{1,2} \\
&= e_1 e_2  [ 1+ (u-1)(e_1+e_2+e_{1,3})+(u-1)^2(u+2)\,e_1e_2 ] g_{1,2}\\
&= [1+ (u-1)(e_1+e_2+e_{1,3})+(u-1)^2(u+2)\,e_1e_2 ]e_1 e_2 g_{1,2}\\
&= [ 1+ (u-1)(e_1+e_2+e_{1,3})+(u-1)^2(u+2)\,e_1e_2 ] r_{1,2}.
\end{align*}
\end{proof}

\begin{lemma}\label{ctllemma}
For the element $c_{1,2}$ we have in ${\rm Y}_{d,n}(u)$:
\[
\begin{array}{crcl}
{\it (1)} &g_1 c_{1,2} &=& [1+(u-1)e_1]c_{1,2}\\
{\it (2)} &g_2 c_{1,2} &=& [1+(u-1)e_2]c_{1,2}\\
{\it (3)} &g_1g_2 c_{1,2} &=& [ 1 +(u-1) e_1+(u-1) e_{1,3}+(u-1)^2 e_1e_2 ]c_{1,2}\\
{\it (4)} & g_2g_1 c_{1,2} &=& [ 1 +(u-1)\, e_2+(u-1) e_{1,3}+(u-1)^2e_1e_2 ] c_{1,2}\\
{\it (5)} & g_1g_2 g_1c_{1,2} &=& [ 1+ (u-1)(e_1+e_2+e_{1,3})+(u-1)^2(u+2)\,e_1e_2 ] c_{1,2}.
\end{array}
\]
\end{lemma}

\begin{proof}
The proof is completely analogous to the proof of Lemma~\ref{condeqsftl}.
\end{proof}

\section{Markov traces}\label{ftlsection}

The main purpose of this section is to find the necessary and sufficient conditions in order that the  trace ${\rm tr}$ defined on  ${\rm Y}_{d,n}(u)$ \cite{ju}  passes to the quotient algebras ${\rm FTL}_{d,n}(u)$ and ${\rm CTL}_{d,n}(u)$. 
Since the defining ideal of  ${\rm FTL}_{d,n}(u)$ (respectively of ${\rm CTL}_{d,n}(u)$) is principal,  by the linearity of ${\rm tr}$, we have that $\rm tr$ passes to  ${\rm FTL}_{d,n}(u)$ (respectively to ${\rm CTL}_{d,n}(u)$) if and only if we have: 
\begin{equation}\label{trwg12}
{\rm tr}(\mathfrak{m}\, r_{1,2})=0 \quad (\text{respectively}\quad  {\rm tr}(\mathfrak{m}\, c_{1,2})=0 ),
\end{equation}
for all  monomials $\mathfrak{m}$ in the inductive basis of ${\rm Y}_{d,n}(u)$. So, we seek necessary and sufficient conditions for Eq.~\ref{trwg12} to hold. The strategy is to find such conditions first for $n=3$ and then to generalize using induction. 

\subsection{{\it Computations on ${\rm tr}$}}Recall that elements in the inductive basis of ${\rm Y}_{d,3}(u)$  are of the following forms:
\begin{equation}\label{basicwords}
t_1^{a}t_2^{b}t_3^c,  \quad t_1^{a}g_1t_1^{b}t_3^c, \quad t_1^{a}t_2^{b}g_2g_1t_1^c, \quad  t_1^{a}t_2^{b}g_2t_2^c,\quad t_1^{a}g_1t_1^{b} g_2t_2^c, \quad t_1^{a}g_1t_1^{b} g_2g_1t_1^c,
\end{equation}
where $0\leq a,b,c \leq d-1$ (see Proposition~\ref{inductiveYH}). 
We need now to compute the trace of the elements $\mathfrak{m}\, r_{1,2}$, where $\mathfrak{m}$ runs the monomials listed in \eqref{basicwords}. To do these  computations  we will use the following lemma and proposition.

\begin{lemma}\label{tracelemma}
For all $0\leq m\leq d-1$, we have:
\[
 {\rm tr}\left(e_1^{(m)} e_2 g_{1,2}\right )= (u+1)z^2x_{m} + (u+2)z \, E^{(m)}  +{\rm tr}(e_1^{(m)}e_2).
\]
\end{lemma}

\begin{proof}
By direct computation we have:
\begin{align*}
&{\rm tr}\left(e_1^{(m)} e_2 g_{1,2}\right) =   {\rm tr}\left(e_1^{(m)} e_2 g_1\right ) + {\rm tr}\left(e_1^{(m)} e_2 g_2 \right) +{\rm tr}\left(e_1^{(m)} e_2 g_1 g_2\right)  \\
&\quad  + {\rm tr}\left(e_1^{(m)} e_2 g_2 g_1\right) + {\rm tr}\left(e_1^{(m)} e_2 g_1 g_2 g_1\right)+  {\rm tr}\left(e_1^{(m)} e_2\right)\\
&= \quad \frac{1}{d^2}\sum_{s=0}^{d-1}\sum_{k=0}^{d-1} {\rm tr}(t_1^{m+s} t_2^{-s+k} t_3^{-k} g_1) + \frac{1}{d^2}\sum_{s=0}^{d-1}\sum_{k=0}^{d-1} {\rm tr}(t_1^{m+s} t_2^{-s+k} t_3^{-k} g_2)\\
& \quad +\frac{1}{d^2}\sum_{s=0}^{d-1}\sum_{k=0}^{d-1} {\rm tr}(t_1^{m+s} t_2^{-s+k} t_3^{-k} g_1 g_2) + \frac{1}{d^2}\sum_{s=0}^{d-1}\sum_{k=0}^{d-1} {\rm tr}(t_1^{m+s} t_2^{-s+k} t_3^{-k} g_2 g_1 ) \\
& \quad + \frac{1}{d^2}\sum_{s=0}^{d-1}\sum_{k=0}^{d-1} {\rm tr}(t_1^{m+s} t_2^{-s+k} t_3^{-k} g_1 g_2 g_1) + {\rm tr}\left(e_1^{(m)} e_2\right)\\
&= 2z E^{(m)} + 2z^2 x_{m}+ +zE^{(m)} +  (u-1) z E^{(m)} + (u-1)z^2 x_{m} \\
&= (u+1) z^2 x_{m} + (u+2) z E^{(m)} + {\rm tr}\left(e_1^{(m)} e_2\right).
\end{align*}
\end{proof}

\begin{proposition}\label{fourcaseslemma}
For all $0\leq a, b, c\leq d-1$, we have:
\begin{enumerate}
\item If $\mathfrak{m}=t_1^{a}t_2^{b}t_3^c$, 
\[
{\rm tr}(\mathfrak{m}r_{1,2}) =(u+1)z^2 x_{a+b+c} + (u+2) E^{(a+b+c)} z + {\rm tr}(e_1^{(a+b+c)}e_2)
\]
\item If $\mathfrak{m}= t_1^{a}g_1t_1^{b}t_3^c$ and $\mathfrak{m}=   t_1^{a}t_2^{b}g_2t_2^c$,  
\[
{\rm tr}(\mathfrak{m}r_{1,2}) = u\left [(u+1)z^2 x_{a+b+c} + (u+2) E^{(a+b+c)} z + {\rm tr}(e_1^{(a+b+c)}e_2)\right]
\]
\item If $\mathfrak{m}= t_1^{a}t_2^{b}g_2g_1t_1^c$ and $\mathfrak{m}= t_1^{a}g_1t_1^{b} g_2t_2^c$, 
\[
{\rm tr}(\mathfrak{m}r_{1,2}) =u^2\left[(u+1)z^2 x_{a+b+c} + (u+2) E^{(a+b+c)} z + {\rm tr}(e_1^{(a+b+c)}e_2)\right]
\]
\item If $\mathfrak{m}=  t_1^{a}g_1t_1^{b} g_2g_1t_1^c$, 
\[
{\rm tr}(\mathfrak{m}r_{1,2}) = u^3\left[(u+1)z^2 x_{a+b+c} + (u+2) E^{(a+b+c)} z + {\rm tr}(e_1^{(a+b+c)}e_2)\right].
\]
\end{enumerate}
\end{proposition}
\begin{proof}
We will prove claim (1). According to Eq. \ref{rij=} we have: 
$
\mathfrak{m}r_{1,2} = t_1^{a}t_2^{b}t_3^cr_{1,2}= t_1^{a}t_2^{b}t_3^ce_1e_2g_{1,2}
$. But $t_1^{a}t_2^{b}t_3^ce_1e_2= e_1^{(a+b+c)}e_2$, hence:
\[
\mathfrak{m}r_{1,2} = e_1^{(a+b+c)}e_2g_{1,2}.
\]
Thus, claim (1) follows by applying Lemma \ref{tracelemma}.

For proving the rest of the claims we use Lemmas \ref{condeqsftl} and \ref{tracelemma} and we follow the same argument, so we finish the proof  of the proposition by proving only one representative case.  We shall prove claim (3)  for $\mathfrak{m}= t_1^a g_1 t_1^b g_2 t_2^c$. This monomial can be rewritten as  $t_1^at_2^b t_3^c g_1 g_2$. Now, by using Lemma \ref{condeqsftl} on $g_1 g_2 r_{1,2}$, we obtain:
\[
\mathfrak{m} r_{1,2} = t_1^at_2^b t_3^c g_1 g_2 r_{1,2} =  t_1^at_2^b t_3^c\left[1 + (u-1)e_1 + (u-1)e_{1,3} + (u-1)^2e_1e_2 \right] r_{1,2},
\]
then using now Eq. \ref{rij=} and the fact the $e_i$'s are idempotents, it follows that:
\begin{align*}
\mathfrak{m} r_{1,2} & =   t_1^at_2^b t_3^c\left[e_1e_2 + (u-1)e_1e_2 + (u-1)e_1e_2 + (u-1)^2e_1e_2 \right] g_{1,2}\\
& =  
u^2 t_1^at_2^b t_3^ce_1e_2 g_{1,2}.
\end{align*}
Then, applying Eq. \ref{eimei+1=} we have:
\[
\mathfrak{m} r_{1,2} = u^2 t_1^at_2^b t_3^ce_1e_2 g_{1,2} =  u^2 e_1^{(a+b+c)}e_2 g_{1,2}.
\]
Therefore, by using Lemma \ref{tracelemma}, we obtain the desired expression for 
${\rm tr}(\mathfrak{m} r_{1,2}) $.
\end{proof}

\subsection{{\it Passing ${\rm tr}$ to the algebra ${\rm YTL}_{d,n}(u)$}}

In \cite{gojukola} we found the necessary and sufficient conditions so that ${\rm tr}$ passes to ${\rm YTL}_{d,n}(u)$. Indeed, we have the following:
 \begin{theorem}[{\cite[Theorem 6]{gojukola}}]\label{thmy}
The trace ${\rm tr}$ passes to the quotient algebra ${\rm YTL}_{d,n}(u)$ if and only if the $x_i$'s  are solutions of the ${\rm E}$--system and one of the two cases holds: 
\begin{enumerate}
\item [(i)] 
the $x_\ell$'s are  $d^{th}$ roots of unity and  $z=-\frac{1}{u+1}$ or $z=-1$,
\item [(ii)] 
the $x_\ell$'s are the solutions of the {\rm E}--system that are parametrized by the set $D=\{ m_1, m_2 \, | \, 0 \leq m_1,m_2 \leq d-1\, \mbox{and } m_1\neq m_2\}$ and they are expressed as: 
\[
x_\ell= \frac{1}{2} \left (\chi_{m_1}(t^\ell) + \chi_{m_2}(t^\ell) \right ),  \quad 0 \leq \ell \leq d-1.
\]
In this case we have that $z=-\frac{1}{2}$.
\end{enumerate}
\end{theorem}

\subsection{{\it Passing ${\rm tr}$ to the algebra ${\rm FTL}_{d,n}(u)$}}
The following lemma is  key to proving one of our main results (Theorem~\ref{akthmgen}).
Recall that the {\it support} of a function $x:C_d \rightarrow \mathbb{C}$ (or equivalently of an element $\sum_{k=0}^{d-1} x(t^k) t^k \in 
\mathbb{C}C_d$) is the subset of $C_d$ where the values of $x$ are non-zero.

\begin{lemma}\label{lemmaN3}
The trace {\rm tr } passes to ${\rm FTL}_{d,3}(u)$ if and only if the parameters of the trace {\rm tr} satisfy: 
\[
x_k = -z \left(\sum_{m\in {\rm Sup}_1}\chi_{ k}(t^{m}) + (u+1)\sum_{m\in {\rm Sup}_2}\chi_{ k}(t^{m}) \right)
\quad \text{and}\quad  
z=-\frac{1}{\vert {\rm Sup_1}\vert + (u+1)\vert {\rm Sup_2}\vert  },
\]
where ${\rm Sup}_1\cup \rm{Sup}_2$ (disjoint union) is the support of the Fourier transform of $x$, and $x$ is the complex function on $C_d$,   
that maps $0$ to $1$ and $k$ to the trace parameter  $x_k$ (cf. Section~\ref{esyssec}).

\end{lemma}
\begin{proof}
Recall that  the trace {\rm tr} passes to ${\rm FTL}_{d,3}$ if and only if the Eqs. \ref{trwg12} hold, for all 
$\mathfrak{m}$ in the inductive basis of ${\rm Y}_{d,3}$. By using Proposition \ref{fourcaseslemma} follows that the trace  {\rm tr} passes to the quotient algebra ${\rm FTL}_{d,3}(u)$ if and only if the  trace parameters $z$, $x_1, \ldots , x_{d-1}$ satisfy the following system of equations:
\[
 \mathbb{E}_0= \mathbb{E}_1 = \cdots =\mathbb{E}_{d-1}= 0 ,
\]
where
\[
\mathbb{E}_m := (u+1)z^2x_m + (u+2) E^{(m)} z +{\rm tr}(e_1^{(m)}e_2) =0,  \quad 0 \leq m \leq d-1 .
\]
We note now that this system of equations above is equivalent to the system:
\begin{equation}\label{traeq2}
\begin{array}{rc} 
\mathbb{E}_0 =  0 & \\
\mathbb{E}_m -x_m\mathbb{E}_0 =0 & \text{where} \quad 1\leq m \leq d-1 .
\end{array}
\end{equation}
We will solve this system of equations, obtaining thus the proof of the lemma.

Recall that  $x_0 :=1$, $E^{(0)} = E$ and $e_i^{(0)} = e_i$, hence 
$
\mathbb{E}_0 =  (u+1)z^2 + 
(u+2) E z + {\rm tr}(e_1e_2)
$. Then the $(d-1)$ equations $\mathbb{E}_m - x_m \mathbb{E}_0 = 0$ of Eq. \ref{traeq2} become: 
\begin{equation} \label{1st-sys}
z(u+2)\big( E^{(m)}-x_m E \big )=- \left (\mathrm{tr} (e_1^{(m)}e_2)-x_m \, {\rm tr}(e_1e_2)  \right), \quad 1\leq m \leq d-1.
\end{equation}

Interpreting now the above equation in the functional notation of Section~\ref{sectioninv} and having in mind Eq. \ref{x*x}, it follows that Eq. \ref{1st-sys}  can be rewritten as:
\[
(u+2) z\left( \frac{1}{d} x\ast x- Ex\right) = -\left( \frac{1}{d^2} x\ast x\ast x- {\rm tr}(e_1e_2)x\right).
\]
Applying now the Fourier transform on the above functional equality and using Proposition \ref{propietrans}, we obtain:
\begin{equation} \label{11}
(u+2)z \left(\frac{\widehat{x}^2}{d}-E\widehat{x}  \right)=
-\left(\frac{\widehat{x}^3}{d^2}-{\rm tr}(e_1e_2)\widehat{x} \right) .
\end{equation}
Let now $\widehat{x}= \sum_{m=0}^{d-1} y_mt^m$. Then Eq. \ref{11} becomes:
\[
(u+2)z  \left(\frac{y_m^2}{d} -Ey_m \right) 
=-
\left(\frac{y_m^3}{d^2}- {\rm tr}(e_1e_2)y_m \right) .
\]
Hence
\begin{equation} \label{12}
y_m \left(
\frac{y_m^2}{d^2} + (u+2)z \frac{y_m}{d}
-(u+2)z E-{\rm tr}(e_1e_2)
\right)=0 .
\end{equation}
Now, from equation $\mathbb{E}_0  =0$, we have that $-(u+2)zE= (u+1)z^2 + {\rm tr}(e_1e_2)  $. Replacing this 
expression of $ -(u+2)zE$ in Eq.~\ref{12} we have that: 
\[
y_m \left(
\frac{y_m^2}{d^2} + (u+2)z \frac{y_m}{d}
+(u+1)z^2
\right)=0,
\]
or equivalently (notice that the equivalence still holds even if we specialize $u=-1$, where the above equation is not quadratic):
\begin{equation}  \label{solve-ell}
y_m \left( y_m + dz \right)
\left( y_m +  dz (u+1) \right)
=0 .
\end{equation}

Denote ${\rm Sup}_1\cup {\rm Sup}_2$ the support of $\widehat{x}$, where 
\[
{\rm Sup}_1:= \{m\in C_d\,; \,y_m = - dz  \} \quad \text{and }\quad  {\rm Sup}_2:=\{m\in C_d\,; \,y_m = - dz(u+1)\},
\]
hence
\[
\widehat{x} = \sum_{m\in {\rm Sup}_1}-dzt^m + \sum_{m\in {\rm Sup}_2}- dz(u+1)t^m.
\]
Notice again that if specialize $u=-1$, then the support of $\widehat{x}$ is just ${\rm Sup}_1$.
Then
\[
\widehat{\widehat{x}} = - dz\sum_{m\in {\rm Sup}_1}\widehat{\delta}_m - dz(u+1)\sum_{m\in {\rm Sup}_2}\widehat{\delta}_m,
\]
thus from argument (4) of Proposition \ref{propietrans} we have: 
\[
\widehat{\widehat{x}} = - z\left( \sum_{m\in {\rm Sup}_1}\mathbf{i}_{-m}  + (u+1)\sum_{m\in {\rm Sup}_2}\mathbf{i}_{-m}\right).
\]
Therefore, having in mind now (5) of Proposition \ref{propietrans}, we deduce that:
\begin{equation}\label{valuesxk}
x_k = -z \left(\sum_{m\in {\rm Sup}_1}\chi_{k}(t^{m}) + (u+1)\sum_{m\in {\rm Sup}_2}\chi_{k}(t^{m}) \right) .
\end{equation} 

Having in mind that $x_0=1$, one can determine the values of $z$. Indeed, from Eq. \ref{valuesxk}, we have that:
\begin{equation} \label{valuesz0}
1= x_0 = -z(\vert {\rm Sup_1}\vert +(u+1) \vert {\rm Sup_2}\vert  ),
\end{equation}
or equivalently  (keep in mind that the assumption $x_0=1$ forces the denominator to be non-zero and hence the support of $\widehat{x}$ is not empty):
\begin{equation}\label{valuesz}
z=-\frac{1}{\vert {\rm Sup_1}\vert +(u+1) \vert {\rm Sup_2}\vert}  .
\end{equation}
By the same reasoning $z$ is also non-zero.
\end{proof}

Keeping the same notation with the above lemma, we have:
\begin{theorem}\label{akthmgen}
The trace ${\rm tr}$ defined on ${\rm Y}_{d,n}(u)$ passes to the quotient algebra ${\rm FTL}_{d,n}(u)$ if and only if the trace parameters $z, x_1 ,\ldots , x_{d-1}$  satisfy the conditions of Lemma \ref{lemmaN3}, namely 
Eqs. \ref{valuesxk} and  \ref{valuesz}.
\end{theorem}
\begin{proof}
The proof is by induction on $n$. The case $n=3$ is the lemma above.  Assume now that the statement holds for all ${\rm FTL}_{d,k}(u)$, where $ k \leq n$, that is:
\[
{\rm tr} ( a_k \, r_{1,2} ) = 0,
\]
for all $a_k \in {\rm Y}_{d,k}(u)$, $k \leq n$. We will show the statement for $k=n+1$. It suffices to prove that the trace vanishes on any element of the form
 $a_{n+1}r_{1,2}$, where $a_{n+1}$ belongs to the inductive basis of ${\rm Y}_{d,n+1}(u)$ (recall Eq.~\ref{inductiveYH}), given the conditions of the theorem. Namely:
\[{\rm tr}(a_{n+1}\, r_{1,2}) =0.\]
Since $a_{n+1}$ is in the inductive basis of ${\rm Y}_{d,n+1}(u)$, it is of one of the following forms:
\[
a_{n+1}=a_{n} g_{n} \ldots g_i t_i^k \quad \mbox{or} \quad a_{n+1}=a_{n}t_{n+1}^k ,
\]
where $a_n$ is in the inductive basis of ${\rm Y}_{d,n}(u)$. For the first case we have:
\[
{\rm tr} (a_{n+1} \, r_{1,2}) = {\rm tr} (a_{n} g_{n} \ldots g_i t_i^k \, r_{1,2})
= z\, {\rm tr}(a_{n} g_{n-1} \ldots g_i t_i^k\, r_{1,2}) =z\, {\rm tr}( w \, r_{1,2}),
\]
where $w:=a_{n} g_{n-1} \ldots g_i t_i^k$. Notice now that $w $ is a word in ${\rm Y}_{d,n}(u)$ and so, by the linearity of the trace, we have that ${\rm tr}(w \,r_{1,2})$ is a linear combination of traces of the form ${\rm tr}(a_{n}\,r_{1,2})$, where $a_{n}$ is in the inductive basis of ${\rm Y}_{d,n}(u)$. Therefore, by the induction hypothesis, we deduce that:
\[ 
{\rm tr}(w \, r_{1,2}) =0,
\]
 if and only if the conditions of the Theorem are satisfied. Therefore the statement is proved. The second case is proved similarly. Hence, the proof is concluded.
\end{proof}
\begin{corollary}\label{coroles}
In the case where one of the sets ${\rm Sup}_1$ or ${\rm Sup}_2$ is the empty set, the values of the $x_k$'s are solutions of the ${\rm E}$--system. More precisely,  if ${\rm Sup}_1$ is the empty set, the $x_k$'s are the solutions of the   ${\rm E}$--system parametrized by ${\rm Sup}_2$ and $z= -1/(u+1)\vert{\rm Sup}_2\vert$. If ${\rm Sup}_2$ is the empty set, then $x_k$'s are the solutions of the ${\rm E}$--system parametrized by ${\rm Sup}_1$ and 
$z= -1/\vert {\rm Sup}_1 \vert$.
\end{corollary}
\begin{proof}
The proof follows from Eq. \ref{solEsys} and the expression given in theorem above for the $x_k$'s.
\end{proof}

\subsection{{\it Passing ${\rm tr}$ to the algebra ${\rm CTL}_{d,n}(u)$}}\label{ctlsection}
The method for finding the necessary and sufficient conditions for ${\rm tr}$ to pass  to the quotient algebra  ${\rm CTL}_{d,n}(u)$ is completely analogous to that of the previous subsection. So, we will need the following analogue of Proposition \ref{fourcaseslemma}.

\begin{proposition}\label{trmc12}
Define $\mathbb{G}$, as follows:
\[
\mathbb{G} = (u+1)z^2 \sum_{k=0}^{d-1} x_k + (u+2)z  \sum_{k=0}^{d-1} E^{(k)} + \sum_{k=0}^{d-1} {\rm tr}(e_1^{(k)}e_2) .
\]
Then for all $0\leq a, b, c\leq d-1$, we have: 
\[
\begin{array}{clll}
{\it (1)} & {\rm tr}( \mathfrak{m}c_{1,2} ) = \mathbb{G}& \mbox{for} &  \mathfrak{m}=t_1^{a}t_2^{b}t_3^c \\
{\it (2)} &{\rm tr}( \mathfrak{m} c_{1,2}) = u \mathbb{G} &\mbox{for} &  \mathfrak{m}= t_1^{a}g_1t_1^{b}t_3^c  \qquad \mbox{and} \quad \mathfrak{m}= t_1^{a}t_2^{b}g_2t_2^c  \\
{\it (3)} &{\rm tr}( \mathfrak{m} c_{1,2}) =
u^2 \mathbb{G} & \mbox{for} & \mathfrak{m}= t_1^{a}t_2^{b}g_2g_1t_1^c \quad \mbox{and} \quad \mathfrak{m}= t_1^{a}g_1t_1^{b} g_2t_2^c\\
{\it (4)} & {\rm tr}( \mathfrak{m}c_{1,2} ) = u^3 \mathbb{G} &\mbox{for}  &\mathfrak{m}=  t_1^{a}g_1t_1^{b} g_2g_1t_1^c. 
\end{array}
\]
\end{proposition}

Following now the analogous reasoning that was used to prove Theorem \ref{akthmgen} and having in mind 
 Eq.~\ref{trwg12}, Corollary \ref{ctlpr},  Lemma \ref{ctllemma} and  Proposition \ref{trmc12},  we obtain  the following theorem.  
\begin{theorem}\label{ctlthm}
The trace ${\rm tr}$ passes to the quotient algebra ${\rm CTL}_{d,n}(u)$ if and only if the parameter $z$ and the $x_i$'s are related through the equation:
 \begin{equation}\label{ctlbasic}
 (u+1)z^2\sum_{k\in \mathbb{Z}/d\mathbb{Z}} x_k  +  (u+2)z\sum_{k\in \mathbb{Z}/d\mathbb{Z}} E^{(k)}+ \sum_{k \in \mathbb{Z}/d\mathbb{Z}}{\rm tr}(e_1^{(k)} e_2)=0.
 \end{equation}
\end{theorem}

\subsection{{\it Comparison of the three trace conditions}}
 In this section we will compare the conditions that need to be applied to the trace paramaters $z$ and $x_i$, $i=1, \ldots , d-1$ so that ${\rm tr}$ passes to each one of the quotient algebras ${\rm YTL}_{d,n}(u)$, ${\rm{FTL}_{d,n}(u)}$ and ${\rm CTL}_{d,n}(u)$. 
 
By  comparing Theorem~\ref{thmy} and Theorem~\ref{akthmgen}, we observe that the conditions such that ${\rm tr}$ passes to ${\rm YTL}_{d,n}(u)$ are contained in the conditions such that ${\rm tr}$ passes to ${\rm FTL}_{d,n}(u)$. 

Moreover, Theorem~\ref{akthmgen} can be rephrased in the following way:
\begin{theorem}\label{thmf}
The trace ${\rm tr}$ passes to the quotient algebra ${\rm FTL}_{d,n}(u)$ if and only if the parameter $z$ and the $x_i$'s are related through the equation:
 \[
 (u+1)z^2 x_k  +  (u+2)z E^{(k)}+{\rm tr}(e_1^{(k)} e_2)=0, \qquad  k \in \mathbb{Z}/d\mathbb{Z}.
 \] 

\end{theorem}

\noindent This implies that the conditions such that the trace passes to the quotient algebra ${\rm FTL}_{d,n}(u)$ are contained in those of Theorem~\ref{ctlthm}. 

 All of the above can be summarized in the following table:
\begin{table}[H]
\[
 \begin{tabular} {l | c  c c c c c c}
  & ${\rm Y}_{d,n}(u)$ & $\twoheadrightarrow$ & ${\rm CTL}_{d,n}(u)$ & $\twoheadrightarrow$ & ${\rm FTL}_{d,n}(u)$ & $\twoheadrightarrow$ & ${\rm YTL}_{d,n}(u)$\\
 \hline
 $z$ &  free& \multirow{2}{*}{$\hookleftarrow$} & \multirow{2}{*}{Theorem~\ref{ctlthm}} & \multirow{2}{*}{$\hookleftarrow$} & \multirow{2}{*}{Theorem~\ref{thmf}} &\multirow{2}{*}{ $\hookleftarrow$} & \multirow{2}{*}{Theorem~\ref{thmy}}\\
 $x_i$ & free &
 \end{tabular} 
\]
\caption{Relations of the algebras and the trace conditions.}
\end{table}

\noindent
The first row includes the projections between the algebras while the second shows the inclusions of the trace conditions for each case.
\begin{remark} \rm
By Theorems~\ref{akthmgen} and \ref{thmf} and by Corollary~\ref{coroles}, the necessary and sufficient conditions for the trace $\rm tr $ to pass to ${\rm FTL}_{d,n}(u)$ include the solutions of the ${\rm E}$--system, leading directly to link invariants derived from this  algebra (see Section~\ref{knotinv}). On the other hand, the conditions on the $x_i$'s for the algebra ${\rm CTL}_{d,n}(u)$ are too loose as indicated by Theorem~\ref{ctlthm}. Moreover, as we shall see in Section~\ref{knotinv} the resulting invariants from ${\rm CTL}_{d,n}(u)$ coincide either with invariants from ${\rm Y}_{d,n}(u)$ or with invariants from ${\rm FTL}_{d,n}(u)$. For these reasons, the algebra ${\rm CTL}_{d,n}(u)$ will be discarded as a possible framization of the Temperley--Lieb algebra.

\end{remark}

\section{Knot invariants}\label{knotinv}
In this section we define framed and classical link invariants related to the algebras ${\rm FTL}_{d,n}(u)$ and ${\rm CTL}_{d,n}(u)$, using the results of the previous sections. The general scheme for defining these invariants follows Jones' method \cite{jo,Homfly}. More precisely, one uses the (framed) braid equivalence corresponding to (framed) link isotopy, the mapping of the (framed) braid group to the knot algebra in question and the Markov trace on this algebra, which, upon re-scaling and normalization according to the braid equivalence, yields isotopy invariants of (framed) links.
\subsection{{\it The Homflypt and the Jones polynomials}}
It is known that by re-scaling and normalizing the Ocneanu trace $\tau$ on ${\rm H}_n(u)$, one can define the 2-variable Jones or Homflypt polynomial, $P(\lambda_{\rm H}, u)$ \cite{jo}. Namely, we have:
\[
P(\lambda_{\rm H},u)(\widehat{\alpha}) = \left( - \frac{1-\lambda_{\rm H} \, u}{\sqrt{\lambda_{\rm H}}(1-u)}\right)^{n-1} \left (\sqrt{\lambda_{\rm H}}\right)^{\varepsilon(\alpha)} {\rm \tau}(\pi(\alpha)),
\]
where: $\alpha \, \in \, \cup_{\infty} B_{n}$, $\lambda_{\rm H} = \frac{1-u+\zeta}{u \zeta}$  is the ``re-scaling factor'', $\pi$ is the natural epimorphism of $\mathbb{C}(u)B_n$ on ${\rm H}_n(u)$ that sends the braid generator $\sigma_i$ to $h_i$ and $\varepsilon(\alpha)$ is the algebraic sum of the exponents of the $\sigma_i$'s in $\alpha$. Further,  by specializing $\zeta$ to $-\frac{1}{u+1}$, the non-trivial value for which the Ocneanu trace $\tau$ passes to the quotient algebra ${\rm TL}_n(u)$, the Jones polynomial, $V(u)$, can be defined through the Homflypt polynomial \cite{jo}, as follows:
\[
V(u)(\widehat{\alpha}) = \left(- \frac{1+u}{\sqrt{u}} \right)^{n-1} \left(\sqrt{u}\right)^{\varepsilon(\alpha)} {\rm \tau}(\pi(\alpha)) = P(u,u)(\widehat{\alpha}).
\]

\subsection{{\it Invariants from ${\rm Y}_{d,n}(u)$}}In \cite{jula} it is proved that the trace ${\rm tr}$ defined on ${\rm Y}_{d,n}(u)$ can be re-scaled according to the braid equivalence corresponding to isotopic framed links if and only if the framing parameters $x_i$'s of ${\rm tr}$ furnish a solution of the ${\rm E}$--system (recall discussion in Section~\ref{sectioninv}). Let  $X_D =( \rm x_1, \ldots , x_{d-1} )$ be a solution of the ${\rm E}$--system parametrized by the non-empty set $D$ of $ \mathbb{Z}/d\mathbb{Z}$. We have the following definition:
\begin{definition}[{\cite[Definition~3 ]{ChLa}}] \rm
The trace map ${\rm tr}_D$ defined as the trace ${\rm tr}$  with the parameters $x_i$ specialized to the values ${\rm x}_i$, shall be called the \emph{specialized trace}  with parameter $z$.
\end{definition}
Note that for $d=1$ the traces ${\rm tr}$ and ${\rm tr}_D$ coincide with the Ocneanu trace. By normalizing ${\rm tr}_D$, an invariant for {\it framed links} is obtained \cite{jula}:
\begin{equation}\label{gammainv}
\Gamma_{d,D}(w,u)(\widehat{\alpha}) = \left(- \frac{(1 - w u)|D|}{\sqrt{w} (1-u) } \right)^{n-1} \left(\sqrt{w}\right)^{\varepsilon(\alpha)} {\rm tr}_D(\gamma(\alpha)),
\end{equation}
where: $w = \frac{z + (1-u)E}{uz}$ is the re-scaling factor, $E=\frac{1}{|D|}$ \cite{jula,jula4}, $\gamma$ is the natural epimorphism of the framed braid group algebra $ \mathbb{C}(u)\mathcal{F}_n$ on the algebra ${\rm Y}_{d,n}(u)$, and $\alpha \in \cup_{\infty} \mathcal{F}_{n}$. 
\smallbreak
 Further, by restricting the invariants $\Gamma_{d,D}(w,u)$ to {\it classical links}, seen as framed links with all framings zero, in \cite{jula4} invariants of classical oriented links $\Delta_{d,D}(w,u)$ are obtained. 

In  \cite{ChLa} it was proved that for generic values of the parameters $u,z$ the invariants $\Delta_{d,D}(w,u)$ do not coincide with the Homflypt polynomial except in the trivial cases $u=1$ and $E=1$. More details are given in Section~\ref{identif}.

\subsection{{\it Invariants from ${\rm YTL}_{d,n}(u)$} }In \cite{gojukola}  the invariants that are defined through the Yokonuma--Temperley--Lieb were studied. More precisely, it was shown that  in order that the trace ${\rm tr}$ passes to the quotient algebra ${\rm YTL}_{d,n}(u)$ it is necessary that the $x_i$'s are $d^{th}$ roots of unity. These furnish a (trivial) solution of the ${\rm E}$--system and in this case $E=1$. By \cite[Remark~5]{jula} the framed link invariants $\mathcal{V}_D(u)$ that are derived from ${\rm YTL}_{d,n}(u)$ are not very interesting, since basic pairs of framed links are not distinguished. On the other hand, by \cite{ChLa}, the classical link invariants $V_D(u)$ that we obtain from ${\rm YTL}_{d,n}(u)$ coincide with the Jones polynomial. This is the main reason that the algebra ${\rm YTL}_{d,n}(u)$ does not qualify for being the framization of the Temperley--Lieb algebra.

\subsection{{\it Invariants from ${\rm FTL}_{d,n}(u)$}}
As it has already been stated, the trace parameters $x_i$ should be solutions of the ${\rm E}$--system so that a link invariant through ${\rm tr}$ is well-defined. Recall that, the conditions of Theorem~\ref{akthmgen} include these solutions for the $x_i$'s. So, in order to define link invariants on the level of the quotient algebra ${\rm FTL}_{d,n}(u)$, we discard any values of the $x_i$'s that do not furnish a solution of the ${\rm E}$--system. For a solution of the ${\rm E}$--system parametrized by $D \subset \mathbb{Z}/d\mathbb{Z}$, using Corollary~\ref{coroles}, we have the following:
\begin{proposition}\label{specialftl}
The specialized trace ${\rm tr}_D$ passes to the quotient algebra ${\rm FTL}_{d,n}(u)$ if and only if:
\[
z= - \frac{1}{(u+1) |D|} \quad \mbox{or} \quad z=-\frac{1}{|D|} .
\]
\end{proposition}

We do not take into consideration the case where $z=-\frac{1}{|D|}$, since important topological information is lost. For example, the trace ${\rm tr}_D$ gives the same value for all even (resp. odd) powers of the $g_i$'s, for $m \in \mathbb{Z}^{>0}$ \cite{jula}:
\begin{equation}\label{evenpower}
{\rm tr}_D(g_i^m) = \left( \frac{u^m -1}{u+1} \right )z + \left( \frac{u^m -1}{u+1} \right )\frac{1}{|D|} +1 \qquad \text{if } m \text{ is even}
\end{equation} 
and
\begin{equation}\label{oddpower}
{\rm tr}_D(g_i^m) = \left( \frac{u^m -1}{u+1} \right ) z + \left( \frac{u^m -1}{u+1} \right ) \frac{1}{|D|} - \frac{1}{|D|} \qquad \text{if } m \text{ is odd,}
\end{equation}
so the corresponding knots and links are not distinguished.

From the remaining case, where the $x_i$'s are solutions of the ${\rm E}$--system and $z=-\frac{1}{(u+1)|D|}$, we deduce for the rescaling factor $w$ that appears in Eq.~\ref{gammainv} that $w=u$. We then have the following definition:
\begin{definition}\label{framinvdef}\rm
Let $X_D$ be a solution of the ${\rm E}$--system, parametrized by the non-empty subset $D$ of $\mathbb{Z}/d\mathbb{Z}$ and let $z = -\frac{1}{(u+1)|D|}$. We obtain from $\Gamma_{d,D}(w,u)$ the following 1-variable framed link invariant:
\[
\begin{array}{crcl}
 \Gamma_{d,D}(u, u)(\widehat{\alpha})&:=& \left (- \frac{(1+ u)|D|}{\sqrt{u} } \right)^{n-1} \left(\sqrt{u}\right)^{\varepsilon(\alpha)} {\rm tr}_D\left(\gamma(\alpha)\right) \label{frjon},
\end{array}
\]
for any $\alpha \in \cup_{\infty}\mathcal{F}_n$. Further, in analogy to  the invariants of $\Gamma_{d,D}(w,u)$, if we restrict to framed links with all framings zero, we obtain from  $\Gamma_{d,D}(u,u)$ an 1-variable invariant of classical links $\Delta_{d,D}(u,u)$.
\end{definition}

 \subsection{{\it Invariants from ${\rm CTL}_{d,n}(u)$}}
 The conditions of Theorem~\ref{ctlthm} do not involve the solutions of the ${\rm E}$--system at all, so in order to obtain a well-defined link invariant on the level of ${\rm CTL}_{d,n}(u)$ we must impose ${\rm E}$--condition on the $x_i$'s.
 Recall that the solutions of the ${\rm E}$--system can be expressed in the form:
\[
x_{_D} = \frac{1}{|D|}\sum_{k \in D} \mathbf{i}_k \in \mathbb{C}C_d ,
\]

\noindent where $\mathbf{i}_k =  \sum_{j=0}^{d-1}\chi_k( t^j)t^j$, $\chi_k$ is the character that
sends $t^m\mapsto \cos\frac{2\pi k m}{d} + i \, \sin\frac{2 \pi k m }{d}$ and $D$ is the subset of $\mathbb{Z}/d\mathbb{Z}$ that parametrizes this solution of the ${\rm E}$--system. Let now  ${\rm \varepsilon}$ be the augmentation function of the group algebra $\mathbb{C}C_d$, sending $\sum_{j=0}^{d-1} x_j t^j$ to $\sum_{j=0}^{d-1} x_j$. We have that:
\begin{equation}\label{ctlvalx}
 {\rm \varepsilon}(x_{_D}) = \frac{1}{|D|} \sum_{k \in D} \varepsilon (\mathbf{i}_k) = \frac{1}{|D|}\sum_{j=0}^{d-1} \sum_{k \in D}\chi_k(t^j) =   \left \{ \begin{array}{ccr} \frac{d}{|D|}, & \mbox{if}& 0\in D \\ 0, & \mbox{if}& 0 \notin D \end{array} \right.  .
 \end{equation}
 
From this we deduce that:
\begin{equation}\label{ctlvale}
\sum_{j =0}^{d-1}E^{(j)}  = {\rm \varepsilon}\left (\frac{x*x}{d}\right) = \frac{1}{d|D|^2}\sum_{k \in D} {\rm \varepsilon}(\mathbf{i}_k *\mathbf{i}_k) =  \frac{1}{|D|^2}\sum_{k \in D} {\rm \varepsilon}(\mathbf{i}_k)  =   \left \{ \begin{array}{ccc} \frac{d}{|D|^2}, & \mbox{if}& 0\in D \\ 0, & \mbox{if}& 0 \notin D \end{array} \right.
\end{equation}
and also that:
\small
\begin{equation}\label{ctlvalt}
\sum_{j=0}^{d-1}{\rm tr}(e_1^{(j)}e_2)  = {\rm \varepsilon}\left (\frac{x*x*x}{d^2}\right) = \frac{1}{d^2|D|^3}\sum_{k \in D} {\rm \varepsilon}(\mathbf{i}_k *\mathbf{i}_k*\mathbf{i}_k) =  \frac{1}{|D|^3}\sum_{k \in D} {\rm \varepsilon}(\mathbf{i}_k)  =   \left \{ \begin{array}{ccc} \frac{d}{|D|^3}, & \mbox{if}& 0\in D \\ 0, & \mbox{if}& 0 \notin D \end{array} \right. .
\end{equation}
\normalsize
\\
 Using now Eqs.~\ref{ctlvalx} -- \ref{ctlvalt}, we have that Eq.~\ref{ctlbasic},for the case where $0 \in D$,  becomes : 
 \begin{equation}\label{specialctl}
  \frac{d}{|D|}\left ((u+1) z^2 + \frac{(u+2)}{|D|}z + \frac{1}{|D|^2} \right )=0 .
 \end{equation}
Notice also that for the case where $0 \notin D$, Eq.~\ref{ctlbasic} vanishes. We thus have the following:
\begin{proposition}\label{ctltwoinvs}
Assume that the $x_i$'s are restricted to solutions of the ${\rm E}$--system. Then, the specialized trace ${\rm tr}_D$ passes to the quotient algebra ${\rm CTL}_{d,n}(u)$ if and only if one of the following cases hold:
\begin{enumerate}[(i)]
\item When $0 \in D$, the trace parameter $z$ takes the values:
\[
z= - \frac{1}{(u+1) |D|} \quad \mbox{or} \quad z=-\frac{1}{|D|}.
\]

\item When $0 \notin D$, the trace parameter $z$ is free.
\end{enumerate} 
 \end{proposition}

We will discuss now the invariants that are derived from the algebras ${\rm CTL}_{d,n}(u)$. \\

\noindent {\it Case (i)} $0\in D$. In this case the values for $z$ in case $(i)$ of Proposition~\ref{ctltwoinvs} coincide with the values for $z$ in Proposition~\ref{specialftl}. Further, much like the case of ${\rm FTL}_{d,n}(u)$, the value $z=-\frac{1}{|D|}$ is not taken into consideration. Thus, the invariants that are obtained from ${\rm tr}_D$ on the level of the quotient algebra ${\rm CTL}_{d,n}(u)$, for subsets $D$ containing zero, coincide with the corresponding invariants $\Gamma_{d,D}(u,u)$ and $\Delta_{d,D}(u,u)$ derived from ${\rm FTL}_{d,n}(u)$, since the conditions that are applied to the trace parameters are the same for both quotient algebras.
\smallbreak
\noindent {\it Case (ii)} $0 \notin D$. In this case $z$ remains an indeterminate. Thus, the only condition that is required so that the trace ${\rm tr}_D$ passes to the quotient algebra, is that the $x_i$'s comprise a solution of the $\rm E$--system. This means that the invariants that are derived from the quotient algebra ${\rm CTL}_{d,n}(u)$, for subsets $D$ not containing zero, coincide with the corresponding invariants $\Gamma_{d,D}(w,u)$ and $\Delta_{d,D}(w,u)$ that are derived from ${\rm Y}_{d,n}(u)$. We have thus proved the following:

\begin{proposition}\label{ctlftly}
 Let $X_D$ be a solution of the ${\rm E}$--system parametrized by $D \subset \mathbb{Z}/d\mathbb{Z}$.  The invariants derived from the algebra ${\rm CTL}_{d,n}(u)$:
\begin{enumerate}[(i)]
\item if $0 \in D$, they coincide with the invariants derived from the algebra ${\rm FTL}_{d,n}(u)$ and
\item  if $0 \notin D$, they coincide with the invariants derived from the algebra ${\rm Y}_{d,n}(u)$.
\end{enumerate}
\end{proposition}

\begin{remark} \rm
As we see from the above, the type of invariants we obtain from the algebra ${\rm CTL}_{d,n}(u)$ depends on whether zero belongs or not to the parametrizing set $D$ of  the specific solution of the ${\rm E}$--system. The intrinsic reason for this peculiar condition on the set $D$ is the fact that when summing up all $n$-roots of unity we get zero, unless $n=1$, see Eq.~\ref{ctlvalx}. 
\end{remark}

\begin{remark} \rm
The results of Propositions~\ref{ctltwoinvs} and~\ref{ctlftly} seem to be in accordance with the recent results of Chlouveraki and Pouchin \cite{ChPou2}, where they prove that the algebra is isomorphic to a direct sum of matrix algebras over tensor products of Temperley--Lieb and Iwahori--Hecke algebras.
\end{remark}

To summarize, the solutions of the ${\rm E}$--system (which are the necessary and sufficient conditions so that topological invariants for framed links can be defined) are included in the conditions of Theorem~\ref{akthmgen}, while for the case of ${\rm CTL}_{d,n}(u)$ we still have to impose them. Even by doing so, this algebra does not deliver any new invariants for (framed) links. This is the main reason that led us to consider the quotient algebra ${\rm FTL}_{d,n}(u)$ as the most natural non-trivial analogue of the Temperley--Lieb algebra in the context of framed links.
\smallbreak
We conclude this section with presenting the following tables that give an overview of the invariants for each quotient algebra:

\begin{center}

\begin{table}[H]
\centering
\begin{tabular}{ c|c |c |c |c | c }
\multirow{2}{*}{$d, |D|>1$}     & \multirow{2}{*}{${\rm Y}_{d,n}(u)$}        & \multicolumn{2}{c|}{${\rm CTL}_{d,n}(u)$}         & \multirow{2}{*}{${\rm FTL}_{d,n}(u)$}    & \multirow{2}{*}{${\rm YTL}_{d,n}(u)$}  \\ \cline{3-4}
                      &                           & \footnotesize{$0\notin D$  }   & \footnotesize{$0\in D$}                       &                         &                       \\ \hline
$\mathcal{F}_{d,n}$                     & $\Gamma_{d,D}(w,u)$                     & $\Gamma_{d,D}(w,u)$ & $\Gamma_{d,D}(u,u)$                   & $\Gamma_{d,D}(u,u)$                   & $-$                     \\ 
\multicolumn{1}{  c |}{$B_n$} & \multicolumn{1}{c |}{$\Delta_{d,D}(w,u)$} & $\Delta_{d,D}(w,u)$     & \multicolumn{1}{c | }{$\Delta_{d,D}(u,u)$} & \multicolumn{1}{c |}{$\Delta_{d,D}(u,u)$} & \multicolumn{1}{c }{$-$}\\ 
\end{tabular}
\vspace{.2cm}
\caption{Overview of the invariants for $|D|>1$.}
\end{table}

\end{center}

\begin{center}
\begin{table}[H]
\centering
\begin{tabular}{ c|c |c |c |c | c }
\multirow{2}{*}{$d, |D|=1$}     & \multirow{2}{*}{${\rm Y}_{d,n}(u)$}        & \multicolumn{2}{c|}{${\rm CTL}_{d,n}(u)$}         & \multirow{2}{*}{${\rm FTL}_{d,n}(u)$}    & \multirow{2}{*}{${\rm YTL}_{d,n}(u)$}  \\ \cline{3-4}
                      &                           & \footnotesize{$0\notin D$}     & \footnotesize{$0\in D$}                       &                         &                       \\ \hline
$\mathcal{F}_{d,n}$                     & $\Gamma_{d,D}(w,u)$                     & $\Gamma_{d,D}(w,u)$ & $\mathcal{V}_D(u)$                   & $\mathcal{V}_D(u)$                   & $\mathcal{V}_D(u)$                     \\ 
\multicolumn{1}{  c |}{$B_n$} & \multicolumn{1}{c |}{$P(\lambda, u)$} & $P(\lambda,u)$     & \multicolumn{1}{c | }{$V_D(u)$} & \multicolumn{1}{c |}{$V_D(u)$} & \multicolumn{1}{c }{$V_D(u)$}\\ 
\end{tabular}
\vspace{.2cm}
\caption{Overview of the invariants for $|D|=1$.}
\end{table}
\end{center}

\section{Identifying the invariants from ${\rm FTL}_{d,n}(u)$ on classical knots and links}\label{identif}

It has been a long standing problem how the classical link invariants derived from the Yokonuma--Hecke algebras compare with other known invariants, especially with the Homflypt polynomial. Finally, in a recent development \cite{ChJuKaLa} it is proved that these invariants are topologically equivalent to the Homflypt polynomial on {\it knots} but not on {\it links}.  For proving these results, a different presentation for the algebra ${\rm Y}_{d,n}(u)$ was considered, leading to classical link invariants denoted by $\Theta_d$. As proved in \cite{ChJuKaLa} the invariants $\Theta_d$ do not depend on the sets $D$, so the notation is simplified.

In order to compare the classical link invariants from the Framization of the Temperley--Lieb algebra with the Jones polynomial we will consider in this section a new presentation for this algebra (according to \cite{ChJuKaLa}) and we will adapt our results so far.

\subsection{{\it A different presentation for ${\rm H}_n(u)$ and ${\rm Y}_{d,n}(u)$}}The Iwahori--Hecke algebra is generated by the elements $h_1^\prime ,\ldots, h_{n-1}^\prime$ satifsfying the relations $h_i^\prime h_j^\prime = h_j^\prime h_i^\prime$, for $|i-j|>1$ and $h_i^\prime h_{i+1}^\prime h_i^\prime =h_{i+1}^\prime h_i^\prime h_{i+1}^\prime$, together with the quadratic relations: $ (h_i^\prime)^2 = 1 + (q-q^{-1}) h^\prime_i$. The transformation from the presentation that was given in Section~\ref{prelim} to this one can be achieved by taking $u=q^2$ and $h_i = q h_i^\prime$ \cite{ChJuKaLa}. Consequently, the defining ideal \eqref{idealrel} for the algebra ${\rm TL}_n(q)$ becomes:
\[
1 + q(h^\prime_1 + h^\prime_2) + q^2 (h^\prime_1 h^\prime_2 + h^\prime_2 h^\prime_1) + q^3 h^\prime_1 h^\prime_2 h^\prime_1.
\]
Further, the Ocneanu trace $\tau$ passes to the quotient algebra for the following values of the trace parameter $\zeta^\prime:$ 
\[
 \zeta^\prime = -\frac{q^{-1}}{q^2+1} \quad \mbox{or} \quad \zeta^\prime = - q^{-1} .
\]

On the other hand, the algebra ${\rm Y}_{d,n}(q)$ is generated by the elements $g^\prime_1 , \ldots , g^\prime_{n-1} , t_1 , \ldots , t_n, $ satistfying the relations (\ref{YH1})-(\ref{YH7})  and the quadratic relations \cite{ChJuKaLa}:

\begin{equation}\label{newquad}
(g^\prime_i)^{\,2} = 1+ (q-q^{-1}) e_i g^\prime_i.
\end{equation}

One can obtain this presentation from the one given in Definition~\ref{yhdefpres} by taking $u=q^2$ and 
\begin{equation}\label{switch}
g_i=g^\prime_i+(q-1)e_ig^\prime_i  \quad \mbox{(or, equivalently, }\, g^\prime_i=g_i+(q^{-1}-1)e_i g_i). 
\end{equation}
Further, on the algebra ${\rm Y}_{d,n}(q)$ a unique Markov trace is defined, analogous to ${\rm tr}$, satisfying the same rules \cite{ChJuKaLa}, for which we retain here the same notation. Note also that, the ${\rm E}$--system remains the same for ${\rm Y}_{d,n}(q)$ so we can talk about the specialized trace ${\rm tr}_D$ \cite{ChJuKaLa}. Consequently, in \cite{ChJuKaLa}, invariants for framed links were derived which restrict to invariants of classical links:
\begin{equation}\label{Thetadinv}
\Theta_d(\lambda_D,q) = \left( -\frac{1-\lambda_D}{\sqrt{\lambda_D}(q-q^{-1})E} \right)^{n-1}\sqrt{\lambda_D}^{\varepsilon(\alpha)}{\rm tr}_D(\delta(\alpha)),
\end{equation}
where $\alpha \in \cup_\infty B_n$, $D$ and $\varepsilon(a)$ are as in Definition~\ref{framinvdef} and $\delta$ is the natural epimorphism $\mathbb{C}(q)B_n \rightarrow {\rm Y}_{d,n}(q)$ and $\lambda_D=\frac{{z'} -(q-q^{-1})E}{z'}$ is the re-scaling factor for the trace ${\rm tr}$.

\subsection{{\it A different presentation for ${\rm FTL}_{d,n}(u)$}} Applying now Equations~\ref{newquad} and \ref{switch} to the defining relation~\eqref{rij=} of the Framization of the Temperley--Lieb algebra, we obtain:
\begin{equation}\label{newftl}
 e_1 e_2 \Big( 1 + q(g^\prime_1 + g^\prime_2) + q^2 (g^\prime_1 g^\prime_2 + g^\prime_2 g^\prime_1) + q^3 g^\prime_1 g^\prime_2 g^\prime_1 \Big) = 0.
 \end{equation}
This gives rise to a new presentation, with parameter $q$, for the  Framization of the Temperley--Lieb algebra, as the quotient of ${\rm Y}_{d,n}(q)$ over the ideal that is generated by the relations \eqref{newftl}. We shall denote this isomorphic algebra by ${\rm FTL}_{d,n}(q)$. 

Given this new presentation for ${\rm FTL}_{d,n}(q)$, the necessary and sufficient conditions of Theorem~\ref{akthmgen} such that the trace ${\rm tr}$ on ${\rm Y}_{d,n}(q)$ passes to the quotient become:
\begin{align}\label{newqvalx}
x^\prime_k = -qz^\prime \left(\sum_{m\in {\rm Sup}_1}\chi_{k}(t^{m}) + (q^2+1)\sum_{m\in {\rm Sup}_2}\chi_{k}(t^{m}) \right),
\end{align}
\begin{align}\label{newqvalz}
 z^\prime=-\frac{1}{q\vert {\rm Sup_1}\vert + q(q^2+1)\vert {\rm Sup_2}\vert  }.
\end{align}
Here we used the symbols $x^\prime_k$ and $z^\prime$ for the trace parameters in order to distinguish them from those of ${\rm FTL}_{d,n}(u)$.
If we choose the $x^\prime_k$'s of Eq.~\ref{newqvalx} to be solutions of the ${\rm E}$--system (by letting either ${\rm Sup}_1$ or ${\rm Sup}_2$ to be the empty set), we obtain (respectively) two values for $z^\prime$, the following:
\begin{equation}\label{qvalzz}
z^\prime=-\frac{q^{-1}E}{q^2+1} \quad \mbox{or} \quad z^\prime=-q^{-1} E,
\end{equation}
and  $z$ and $z^\prime$ are related through the equation: $z = q z^\prime$. Using the same arguments as in Section~\ref{knotinv}, the value $z^\prime = -q^{-1} E$ is discarded, since it is of no topological interest. Thus, by specializing in $\Theta_d(\lambda_D,q)$ the trace parameter $z^\prime = -\frac{q^{-1}E}{q^2+1}$ and choosing the $x^\prime_k$'s to be solutions of the {\rm E}--system, we obtain the invariants for classical knots and links:
\begin{equation}\label{thetaq}
\theta_d(q)(\widehat{\alpha}) := \left ( -\frac{1+q^2}{qE} \right)^{n-1} q^{2\varepsilon(\alpha)} {\rm tr}_D (\delta(a)) = \Theta_d(q, q^4)(\widehat{\alpha}),
\end{equation}
where $\alpha \in \cup_\infty B_n$ and $D$, $\lambda_D$, $\varepsilon(a)$ and $\delta$ are as in Eq.~\ref{Thetadinv}
By choosing the values mentioned above for the trace parameters $z^\prime$ and $x^\prime_k$, $0\leq k \leq d-1$, we obtain $\lambda_D=q^4$. 

\subsection{{\it Identification on knots}} For the case of braids in $\cup_{n} B_n$, whose closure is a {\it knot}, the results of \cite{ChJuKaLa} adapt to the following:
\begin{proposition}
The invariants $\theta_d$ are topologically equivalent to the Jones polynomial on knots.
\end{proposition}
\begin{proof}
Let $\alpha \in B_n$ such that its closure $\widehat{\alpha}$ is a knot. By \cite[Theorem~5.17]{ChJuKaLa} we have that:
\begin{equation}\label{thetaeq}
\Theta_d(q)(\widehat{\alpha}) = \Theta_1(q, \lambda_D^{z^\prime/E}) (\widehat{\alpha})= P(q, \lambda_{\rm H}^{z^\prime/E})(\widehat{\alpha}),
\end{equation}
where  $\lambda_D^{z^\prime/E}$ (resp. $\lambda_{\rm H}^{z^\prime/E}$) is the re-scaling factor $\lambda_D$ (resp. $\lambda_{\rm H}$) with the trace parameter $z^\prime$ (resp. $\zeta^\prime$) specialized to $z^\prime/E$.

 Notice now that: $\frac{z^\prime}{E} =-\frac{q^{-1}}{q^2+1}$, which is the value of $\zeta^\prime $ for which the Ocneanu trace $\tau$ passes to the algebra ${\rm TL}_n(q)$.  This implies that: $\lambda_D^{z^\prime/E} = q^4= \lambda_{\rm H}^{z^\prime/E}$.
Thus, Eq.~\ref{thetaeq} becomes:
\[
\theta_d(q)(\widehat{\alpha}) = \theta_1(q, \lambda_D^{z^\prime/E}) (\widehat{\alpha})= P(q, \lambda_D^{z^\prime/E})(\widehat{\alpha})=P(q, q^4)(\widehat{\alpha}) = V(q)(\widehat{\alpha}).
\]
\end{proof}
\subsection{{\it Identification on links}} For the case of {\it classical links}, we work as follows. In \cite{ChJuKaLa}, using data from \cite{ChaLi}, it was observed that, out of 89 pairs of non-isotopic links, which have the same Homflypt polynomial, there are 6 pairs that are distinguished by the invariants $\Theta_d(\lambda_D, q)$. More precisely, the differences of the polynomials for each pair of links were computed and were found to be non-zero. Indeed:

{\small \begin{align*}
&\Theta_d(L11n358\{0,1\})-\Theta_d(L11n418\{0,0\}) = \frac{(E-1) (\lambda_D -1) (q-1)^2 (q+1)^2 \left(q^2-\lambda_D \right)
   \left(\lambda_D  q^2-1\right)}{E \lambda_D^4 q^4},
\\
&\Theta_d(L11a467\{0,1\})-\Theta_d(L11a527\{0,0\}) = \frac{(E-1) (\lambda_D -1) (q-1)^2 (q+1)^2 \left(q^2-\lambda_D \right)
   \left(\lambda_D q^2-1\right)}{E \lambda_D^4 q^4},
\\
&\Theta_d(L11n325\{1,1\})-\Theta_d(L11n424\{0,0\}) = -\frac{(E-1) (\lambda_D -1) (q-1)^2 (q+1)^2 \left(q^2-\lambda_D \right)
   \left(\lambda_D q^2-1\right)}{E \lambda_D ^3 q^4},
\\
&\Theta_d(L10n79\{1,1\})-\Theta_d(L10n95\{1,0\}) = \frac{(E-1) (\lambda_D -1) (q-1)^2 (q+1)^2 \left(\lambda_D +\lambda_D
   q^4+\lambda_D  q^2-q^2\right)}{E \lambda_D^4 q^4},
\\
&\Theta_d(L11a404\{1,1\})-\Theta_d(L11a428\{0,1\}) = \frac{(E-1) (\lambda_D -1) (\lambda_D+1) (q-1)^2 (q+1)^2
   \left(q^4-\lambda_D  q^2+1\right)}{E q^4},
\\  
&\Theta_d(L10n76\{1,1\})-\Theta_d(L11n425\{1,0\}) =  \frac{(E-1) (\lambda_D -1) (\lambda_D+1) (q-1)^2 (q+1)^2}{E \lambda_D^3 q^2}.
\end{align*}}
Note that the factor $(E -1)$, that is common in all six pairs, suggests that the pairs have the same Homflypt polynomial, since for $E=1$ the difference collapses to zero. Further, in \cite{ChJuKaLa} the values of $\Theta_d$ were computed theoretically for one of the six pairs, using a {\it special skein relation} satisfied by $\Theta_d$. Namely, as it is shown in \cite{ChJuKaLa},  the invariants $\Theta_d$ satisfy the Homflypt skein relation, but only for crossings between different components. 
\smallbreak
For  the invariants $\theta_d$, we specialize in the above computations $z^\prime = - \frac{q^{-1}E}{q^2+1}$ (which implies $\lambda_d=q^4$). Clearly, for $E\neq 1$ the six pairs of links above are also distinguished by the invariants $\theta_d$. Moreover, the special skein relation of $\Theta_d$ is also valid for the invariants $\theta_d$, specializing to the following:
\[ 
q^{-2}\, \theta_d\, (L_{+}) - q^2\, \theta_d\, (L_{-}) = (q-q^{-1})\, \theta_d\, (L_{0}),
\]
where the oriented links $L_{+}$, $L_{-}$, $L_{0}$ comprise a Conway triple involving a crossing between different components. From the above, we have thus proved the following:

\begin{theorem}
For $d\neq 1$ and $E\neq 1$, the invariants for classical links $\theta_d(q)$ are not topologically equivalent to the Jones polynomial.
\end{theorem}

\subsection{{\it Concluding notes}}

 The link invariants from the algebras ${\rm FTL}_{d,n}(u)$ still remain under investigation. In this paper the invariants from ${\rm FTL}_{d,n}(q)$ have been compared to the Jones polynomial and have been proved to be topologically non-equivalent. So the related framed link invariants might lead to new 3-manifold invariants analogous to the Witten invariants.  Note that in the case of the algebras ${\rm YTL}_{d,n}(u)$ the Witten invariants only can be recovered, since the related link  invariants recover the Jones polynomial \cite{gojukola}.

\end{document}